\documentclass[11pt]{article}
\usepackage[a4paper,
            bindingoffset=0.2in,
            left=0.75in,
            right=0.8in,
            top=1in,
            bottom=1.5in,
            footskip=.5in]{geometry}
\usepackage{setspace}
\usepackage{color}
\definecolor{aleacolour}{rgb}{0.09,0.32,0.44} 
\usepackage[pdfborder={0 0 0},pdfborderstyle={},colorlinks=true,linkcolor=aleacolour,citecolor=red,urlcolor=blue]{hyperref}
\usepackage{setspace, ifdraft, blindtext,comment}
\ifdraft{\doublespace}{\singlespace} % draft: double; final: single spacing
\ifdraft{%
    \usepackage{todonotes}}{ % draft with TODOs
    \usepackage[disable]{todonotes} % final without TODOs
}

\usepackage{mathtools}
\usepackage{amsmath}
\usepackage{cases}
\usepackage{float}

\usepackage[utf8]{inputenc}
\usepackage{kantlipsum}
\usepackage{tikz}
\usepackage{pgf,tikz,svg}
\usetikzlibrary{arrows}
\usetikzlibrary[patterns]

\usepackage{amsmath, amssymb, amsthm}
\usepackage{color}

\theoremstyle{plain}
\newtheorem{thm}{Theorem}[section]
\newtheorem{prop}[thm]{Proposition}
\newtheorem*{peel}{Peeling Process}
\newtheorem*{prop*}{Proposition}
\newtheorem{lem}[thm]{Lemma}
\newtheorem*{lem*}{Lemma}

\newtheorem{cor}[thm]{Corollary}
\newtheorem{claim}[thm]{Claim}
\newtheorem{fact}[thm]{Fact}

\newtheorem{conj}[thm]{Conjecture}

\newtheorem{ques}[thm]{Question}
\newtheorem*{rem}{Remark}
\theoremstyle{definition}
\newtheorem{definition}[thm]{Definition}


\newcommand{\BD}{\Bar{\Delta}}
\newcommand{\beq}{\begin{equation}}
\newcommand{\eeq}{\end{equation}}
\date{}
\title{\vspace{-0.7cm} {Extremal number of cliques of given orders in graphs with a forbidden clique minor }}

\author{
Ruilin Shi\thanks{School of Mathematics, Georgia Institute of Technology, Atlanta, GA 30318, USA. 
 Email: \href{mailto:rshi49@gatech.edu} {\nolinkurl{rshi49@gatech.edu}}.}
\and 
Fan Wei\thanks{Department of Mathematics, Duke University, Durham, NC 27710, USA. NSF grants DMS-2404167 and DMS-2401414. Email: \href{mailto:fw97@math.duke.edu}
{\nolinkurl{fw97@math.duke.edu}}.}
}

\begin{document}

\maketitle
%\tableofcontents

\begin{abstract}
Alon and  Shikhelman initiated the systematic study of a generalization of the extremal function. Motivated by algorithmic applications, the study of the extremal function $\text{ex}(n, K_k, K_t\text{-minor})$, i.e., the number of cliques of order $k$ in $K_t$-minor free graphs on $n$ vertices, has received much attention. In this paper, we determine essentially sharp bounds on the maximum possible number of cliques of order $k$ in a $K_t$-minor free graph on $n$ vertices. More precisely, we determine a function $C(k,t)$ such that for each $k < t$ with $t-k\gg \log_2 t$, every $K_t$-minor free graph on $n$ vertices has at most $ n C(k, t)^{1+o_t(1)}$ cliques of order $k$. We also show this bound is sharp by constructing  $K_t$-minor-free graph on $n$ vertices with $C(k, t) n$ cliques of order $k$. This bound answers a question of Wood \cite{Wo1} and Fox-Wei \cite{FW} asymptotically up to $o_t(1)$ in the exponent except the extreme values when $k$ is very close to $t$.

\end{abstract}

\section{Introduction}
A clique is a set of vertices where there are edges between any two vertices. We use $K_t$ to denote a clique on $t$ vertices, i.e., of order $t$. We also call it a $t$-clique. 

A cornerstone result in extremal combinatorics is Tur\'an's theorem \cite{turan}, which asks the maximum number of edges in a graph on $n$ vertices that do not have $K_t$ as a subgraph. The answer is obtained by the Tur\'an graph $T(n, t-1)$, which is the complete multipartite graph where each part has order $\lfloor n/(t-1) \rfloor$ or $\lceil n/(t-1) \rceil$. A natural question to ask is: for each positive integer $k < t$, what is the maximum number of cliques of order $k$ in a graph on $n$ vertices without $K_t$ as a subgraph? This is answered by Zykov \cite{zykov}; the same Tur\'an graph $T(n, t-1)$ also maximizes the number of $k$-cliques, i.e., cliques of order $k$. 

Alon and  Shikhelman  \cite{NS} initiated the systematic study of a generalization of this question. Let ${\text{ex}}(n, T, H)$ be  
the maximum possible number of copies of $T$ in an $H$-free graph on $n$ vertices. Thus Tur\'an's theorem gives an answer to ${\text{ex}}(n, K_2, K_t)$ and Zykov's theorem provides an answer to ${\text{ex}}(n, K_k, K_t)$ and furthermore ${\text{ex}}(n, \text{clique}, K_t)$.  Some other examples of results in this trend can be found in \cite{Erdos, Boll, BG, Hatami, NS}.

Analogous questions for forbidding minors have also been studied for a long time, where minors can be considered as a generalization of subgraphs. A graph $H$ is
a \emph{minor} of a graph $G$ if it can be obtained from $G$ by contracting edges and deleting
vertices and edges. A natural generalization asks: what is the maximum possible number of cliques (of possibly fixed sizes) a graph on $n$ vertices could have?  

The study of bounding the number of cliques in $K_t$-minor free graphs, i.e., understanding the extremal functions ${\text{ex}}(n, \text{clique}, K_t\text{-minor})$ and ${\text{ex}}(n, K_k, K_t\text{-minor})$, 
also have applications in theoretical computer science such as designing linear-time algorithms (e.g., see \cite{RW, DKT} and the references therein). 
The bounds on the function ${\text{ex}}(n, \text{clique}, K_t\text{-minor})$  have been studied through works such as by Norine, Seymour, Thomas, and Wollan \cite{NSTW}, Reed and Wood \cite{RW}, Fomin, Oum,  and Thilikos \cite{FOT}, Lee and Oum \cite{LO}, and Wood \cite{Wo1}. 

The paper \cite{NSTW} showed a classical result that the number of $n$-vertex graphs in a proper minor-closed family $\mathcal{I}_n$ is most $c^nn!$ for some constant $c$. 
The proof is through induction by showing that by deleting a twin vertex or by contracting two adjacent vertices with small degrees, there is a mapping from $\mathcal{I}_n$ to $\mathcal{I}_{n-1}$ where the size of pre-image is small. To show this, one key step is to upper bound the number of cliques in $K_t$-minor free graphs.   
The bound on the number of cliques in $K_t$-minor free graphs is later improved to $2^{ct\sqrt{\log t}}n$ by Reed and Wood \cite{RW} by showing that the number of $k$-cliques in $d$-degenarated graph is at most $d^{k-1}n$.

Fomin, Oum, and Thilikos \cite{FOT}  showed more applications of counting cliques in $K_t$-minor free graphs. They 
bounded the tree-width and clique-width of $G$ by the rank-width of $G$ and the number of cliques in $G$, and 
showed that numbers of many important structures are highly related to the number of cliques such as the number of hyperedges in a hypergraph and the number of distinct columns in a binary matrix.   
Notice that they improved the bound of the number of cliques to $2^{ct\log\log t}n$ by bounding the number of $k$-cliques for each $k\le t-1$.

Lee and Oum \cite{LO} considered the number of cliques in $K_t$-subdivision free graphs, and improved the bound to $2^{5t+o(t)}$. 
Wood \cite{Wo1} counted the exact numbers of cliques in the $K_t$-minor free graphs for every $3\le t \le 9$. More precisely, he counted numbers of $k$-cliques in the $K_t$-minor free graphs for every $3\le k< t\le 9$  and gave an upper bound for ${\text{ex}}(n, K_k, K_t)$. He also made several conjectures about this bound which inspired this paper.

The question about the {\it total} number of cliques in $K_t$-minor free graphs was answered by Fox and Wei \cite{FW} where the asymptotically sharp bound is obtained.  
\begin{thm}[Theorem 1.1 \cite{FW} 2016]\label{FWthm}
Every graph on $n$ vertices with no $K_t$-minor has at most $3^{2t/3+o(t)}
n$ cliques.
This bound is tight for $n \geq 4t/3.$
\end{thm}
Note the bound above is adding up the number of cliques of {\it all possible} sizes. 
This bound is asymptotically sharp for $n \geq 4t/3$ by considering a disjoint union of copies of the graph which is the complement of a perfect
matching on $2 \lceil 2t/3 \rceil -2$ vertices. 
Counting the number of cliques was also studied in other graph families that can be found in \cite{KW1, FWimmersion, GHMTZ}.

When we fix the clique size $k$, counting the number of $k$-cliques instead of the total number of cliques in graphs on $n$ vertices with no $K_t$-minor, i.e., to understand ${\text{ex}}(n, K_k, K_t\text{-minor})$,  has received much attention. 
 Clearly, when $n < t$,  the maximum number of cliques of order $k$ is at most $\binom{n}{k}$; this bound is exact and sharp by considering a clique on $n$ vertices, which has no $K_t$-minor. When $k > t$, clearly the answer is 0. The question is less clear for other values of $k$. This thread dates back to the works of Dirac \cite{Dir2}, Mader \cite{Ma1}, J{\o}rgensen \cite{Jor}, and Song and Thomas \cite{ST} for the cases when $k=2$ and $t \leq 9$.  

For general $t$ and any $k < t$, Wood \cite{Wo1} asked the following question, which was asked again by Fox and Wei \cite{FW}.
\begin{ques}[Wood \cite{Wo1}, Fox and Wei \cite{FW}]\label{ques}
    Let $t, k$ be positive integers such that $k < t$. What is the maximum possible number of cliques of order $k$ in a $K_t$-minor free graph on $n$ vertices?
\end{ques}

For small values of $t$, Wood \cite{Wo1} determined the exact value of ${\text{ex}}(n, K_k, K_t\text{-minor})$ for $t \leq 9$ and $k < t$. On the other hand, for larger values of $t$ but for $k=2$, 
 the asymptotic sharp (in $t$) answer is now known after a series of works by Mader, Kostochka, and Thomason  \cite{Ma,Ma1, Ko, Ko1, Th, Th1}. In particular, 

 Kostochka \cite{Ko,Ko1} and Thomason \cite{Th} independently proved that the maximum number of edges in graphs on $n$ vertices and with no $K_t$-minor is $\Theta(t\sqrt{\log_2 t})n$. 
Thomason \cite{Th1} later determines the constant $(\alpha+o_t(1))t\sqrt{\ln t}\cdot n$ where $\alpha=0.319...$ is an explicit constant. This asymptotic extremal configuration can be achieved by random graph $G(n', p')$ with appropriate values of $n'$ and $ p'$.

For larger values of $k$, it seems pseudorandom graphs are no longer optimal. As observed by Fox and Wei \cite{FW},
the average order of the
cliques in the complement of a perfect matching of  $x$ edges is $2x/3$, and thus a typical random clique
in this graph has about this size. Now consider the graph which is a complement of a perfect matching of just less than $2t/3$ edges and is thus $K_t$-free. It has
nearly the maximum number of $k$-cliques for $k = 4t/9$, which gives the $4t/9$-clique count $3^{2t/3 - o(t)} n$. A complement of a perfect matching can be considered as an example of a \emph{Tur\'an graph} that each part has size 2. In general, a candidate for lower bound construction is based on Tur\'an graphs. 

Let $T(n, \omega)$ be the  \emph{Tur\'an graph}, the complete balanced multipartite graph on $n$ vertices and with $\omega$ parts, where each part has order $\lfloor n/\omega \rfloor$ or $\lceil n/\omega \rceil$. Are disjoint unions of Tur\'an graphs nearly optimal? When $k = t-1$,
 Wood \cite{Wo1} shows that the maximum number of $K_{t-1}$ in a $K_t$-minor free graph is exactly $n -t +2$. The construction is called an $(t-2)$-tree (Definition \ref{$(t-2)$-tree}), which is essentially similar to a disjoint union of copies of $K_{t-1}$ where the different copies of $K_{t-1}$ are glued along the same $K_{t-2}$.

 The discussion above shows that
depending on the range of $k$, the extremal constructions for the exact maximum number of $k$-cliques may have quite different forms. We are interested in the asymptotically sharp bounds for the number of $k$-cliques in graphs on $n$ vertices and without $K_t$-minor, where the asymptotic is up to $o(1)$ in the exponent, similar to what asymptotic means as in Theorem \ref{FWthm} \cite{FW}.

Some general upper bounds for this quantity are known. The following simple upper bound is well-known, for example
by Wood \cite{Wo1} Lemma 18, the proof of Norine et al. \cite{NSTW},  the proof of Lemma 3.1 in Reed and Wood \cite{RW}; the proof of Lemma 5 in Fomin et al. \cite{FOT}, or a simplified proof of Theorem 1.1 in Fox and Wei \cite{FW}. 
\begin{thm}[\cite{Wo1, NSTW, RW, FOT, FW}]\label{thm:smallkold}\label{thm:crudeub}
  When $t$ is sufficiently large, for any $k <t$, every graph on $n$ vertices with no $K_t$-minor has at most 
$  \binom{\beta t\sqrt{\ln t}}{k-1} n,$
 cliques of order $k$. The constant $\beta=0.64$. Notice that $\beta > 2\alpha$ where constant $\alpha = 0.319...$ is determined by Thomason \cite{Th1}. 
\end{thm} 
This bound is sharp for $k = 2$ up to a multiplicative constant by the aforementioned result of Thomason \cite{Th1} and by considering a disjoint union of random graphs of appropriate sizes.

Besides this upper bound, Wood \cite{Wo1} made an explicit conjecture on the maximum number of $k$-cliques in $K_t$-minor free graphs on $n$ vertices for large $k$. 
 \begin{conj}[Wood \cite{Wo1} Conjecture 20]  \label{conj:wood2}
For some $\lambda \in [1/3, 1)$,  for all integers $t>3$ and $k>\lambda t$ and $n>t -1$,the number of $k$-cliques in a $K_t$-minor free graph on $n$ vertices is at most 
\[ {t-2 \choose k}+(n-t+2){t-2 \choose k-1}=\binom{t-2}{k-1} \left( n - \frac{(k-1)(t-1)}{k} \right). \] 
\end{conj}
Again, the upper bound is achieved by the $(t-2)$-trees defined below. 
We will prove an asymptotic version of this conjecture for $\lambda > 2/3$ in Corollary \ref{cor: very large k}, and show that the claim of this conjecture does not hold for $\lambda < 0.553$.

 \begin{definition}
[$(t-2)$-tree] \label{$(t-2)$-tree} 
   An \emph{$(t-2)$-tree} is a family of graphs defined recursively as follows: We start with the complete graph $K_{t-2}$,  which is also an $(t-2)$-tree. For any $(t-2)$-tree $H$, if $C$ is a clique of order $t-2$ in $H$, then by adding another vertex to $H$ that is adjacent only to the vertices in $C$ is also an $(t-2)$-tree. Then the number of cliques of order $k$ in every graph in the $(t-2)$-tree family is ${t-2 \choose k}+(n-t+2){t-2 \choose k-1}$.      
 \end{definition} 

 For general values of $k< t$ such that $t-k\gg \log_2 t$, we prove asymptotically sharp bounds on the maximum possible number of $k$-cliques in $K_t$-minor free graphs on $n$ vertices in Theorem \ref{thm: summary}. Again asymptotic here means up to $o_t(1)$ in the exponent, similar to Theorem \ref{FWthm} \cite{FW}.

The main results of this paper are summarized in the next subsection. 
\subsection{Our Results}

In the following main theorem, we answer Question \ref{ques} (Wood \cite{Wo1}, Fox and Wei \cite{FW}) 
up to $o_t(1)$ in the exponent, similar to what asymptotically sharp means as in Theorem \ref{FWthm} \cite{FW}. In other words, we prove a sharp upper bound for the maximum number of cliques of size $k$ in $K_t$-minor free graphs up to $o_t(1)$ in the exponent. %This results holds for all values of $k$ as long as $t- k\ll \log_2t$.   {\color{red}FW: recheck.}

\begin{definition}
    For fixed $k<t$, let $T^*_t(k)$ be the Tur\'an graph  $T(2t-\omega-1,\omega)$ maximizing the number of cliques of order $k$ %maximizing the value of ${2t-\omega-1 \choose k}_\omega$ 
among all $\omega$ such that $k\le \omega \le t-1$. Let $C^*_t(k)$ denoted the number of cliques of order $k$ in $T^*_t(k)$.
\end{definition}

 We will show that $T^*_t(k)$ is $K_t$-minor free for every $t$ and $k<t$ in {Lemma \ref{lem: kt minor free}}. Our main result is the following theorem.

\begin{thm} \label{thm: summary}
 Assume $t-k \gg \log_2 t$. The number of cliques of order $k$ in a $K_t$-minor free graph on $n$ vertices is at most 
  
    \[
     n\cdot \left(\frac{C^*_t(k)}{|V(T^*_t(k))|}\right)^{1+o_t(1)}
    \] 
This bound is sharp up to $o_t(1)$ in the exponent when $n \geq 2t$.
   \end{thm}

 The matching lower bound construction is by considering $\lfloor n/|V(T^*_t(k))| \rfloor$ disjoint copies of the Tur\'an graph $T^*_t(k)$.

\begin{rem}
We have discussed that when $k=2$, pseudorandom graphs are asymptotically optimal \cite{Th1}. It turns out that when $k \ll \log \log t$, the random graph construction also matches the bound in Theorem \ref{thm: summary}, with a slightly better error bound $o_t(1)$ compared to the Tur\'an graph construction. 

\end{rem}

\begin{rem}
    When $k> 2t/3$, Lemma \ref{lem:largekdense} will show that $T^*_t(k)=K_t^-$, the complete graph $K_t$ delete a single edge, and thus $C^*_t(k) = \binom{t-1}{k} + \binom{t-2}{k-1}$.
\end{rem}

\begin{rem}
To see the quantitative behavior of $\frac{C^*_t(k)}{|V(T^*_t(k))|}$ for general values of $k$, first notice  $t \leq |T^*_t(k)| \leq 2t$ by the  definition of $T^*_t(k)$. 
    We will also show that $\binom{t-1}{k} \max\left(1, \left( 2-4\sqrt{ {2k}/{t}}  \right)\right)^k \leq  C^*_t(k) \leq \binom{t-1}{k} 2^k$ for $k \geq 25$ in Claim \ref{claim:simpleckbound} and Lemma \ref{lem:ckapprox}. In addition, $T^*_t(k)$ has $\omega$ parts where $\sqrt{tk}/4\le \omega \le 10 \sqrt{tk}$ as shown in Proposition \ref{prop:Tstarstructure}. 
\end{rem}

We also prove the asymptotic version of Wood's Conjecture \ref{conj:wood2} for every $k< t$ such that $t-k\gg \log_2 t$. In Theorem \ref{thm:woodconj2false} we show that the conjecture is false for $k \leq 0.553 t$.

Notice that the known upper bound in Theorem \ref{thm:smallkold} (Wood \cite{Wo1} and Fox-Wei \cite{FW}) is already an asymptotically sharp bound in the sense above when $k < t^{1-\delta}$ for some absolute constant $\delta$. (For more computational detail see the proof of Theorem \ref{thm: summary} in Section \ref{sec:kmiddle}.) However, not only have we improved this bound for $k$ in this range, but also showed that the new bounds are asymptotically sharp up to $o(1)$ in the exponent for all $k$ such that $t -k \gg O(\log t)$.

For general values of $k,t$, it is challenging to write down a closed-formula description of $T_t^*(k)$. Later Lemma \ref{lem:smalla} tells us that in $T_t^*(k)$, the order of each part is smaller than $\sqrt{\frac{4n-3k}{k}} + 1$; Thus when $k > 4n/7$, the complement of the optimal graph $T_t^*(k)$ is a perfect matching with possibly isolated vertices. It still remains open what the exact description of $T_t^*(k)$ is for general $k$.
The order of each part changes as a function of $k$. 
We could prove an asymptotic result on the size of each part.
\begin{prop}\label{prop:Tstarstructure}
For every $t>k\ge 1$, the optimal $T_t^*(k)$ is given by the Tur\'an graph $T(n,r)$ with $n + r = 2t-1$ where the number of part $r$ satisfies $\sqrt{tk}/4 \le r \le 10 \sqrt{tk}$. When $k \geq   2t/3$, the graph $T_t^*(k)$ is the Tur\'an graph $T(t, t-1)=K^-_{t}$. 
\end{prop}

\paragraph{Proof Idea} 
The proof idea started with a peeling process to encode all cliques of order $K$. This peeling process was used in \cite{FW}, which was highly inspired by the classic paper of Kleitman--Winston \cite{KW}. Roughly speaking, the peeling process maps each clique $K$ into a short encoding $I(K)$ and a ``dense" graph. The authors in \cite{FW} showed that the number of encoding $|\{I(K): K\subset G \}|$ is small, and it is relatively easier to bound the size of the maximum clique minor in a dense graph. 
However, as observed in \cite{FW}, even though the method could provide an almost sharp bound on the total number of cliques in $G$, it fails to provide a satisfactory answer when we fix the clique size $k$. The challenges are two-fold: first, the upper bound on the number of encodings proved in \cite{FW} could be too large for some ranges of $k$, and also we need to characterize the optimal dense graphs optimizing the number of $K_k$.  
In this paper, we made three improvements to overcome the difficulties. The first is that, by a careful analysis of the peeling process, we show that, if $I(K)$ is large, either the number of such encoding is small, or we can find a much bigger clique minor in $G$ which would lead to a contradiction. This idea is particularly important when $k$ is in extreme ranges. 
The second improvement is made by showing a better upper bound for the number of possible representations $I(K)$ when fixing some parameters of $I(K)$. The third is a different method to bound the number of cliques of a given size in the dense graph.

\subsection{Organization of the paper}

In Section \ref{sec:container}, we will define the peeling process as in \cite{FW}, and then prove two key lemmas: one will illustrate how to reduce our problem to the case when the graphs are ``dense" with a better error bound (Key Lemma 2); and another lemma will lower bound the size of the maximum clique minor in the ``dense" graphs given by parameters of the peeling process (Key Lemma 1).

 In Section \ref{sec:kmiddle}, we will prove our main Theorem \ref{thm: summary}  for $k$ in three different ranges. 
 The first two ranges are for $k$ very large range, i.e., $k \geq 2t/3+2\sqrt{t}{\log_2}^{1/4} t$ (Theorem \ref{thm:klarge0}); and for $k$ moderately large, i.e., $\min (k,t-k) \gg O(t^{1/2}{\log_2}^{5/4} t)$ 
 (Theorem \ref{thm:main}). In these two cases, we will also prove  Lemma \ref{lem:largekdense} and Proposition \ref{prop:strucmiddle} (Proposition \ref{prop:Tstarstructure}) which illustrate the structure of $T^*_t(k)$ and help bound the value of $\frac{C^*_t(k)}{|V(T^*_t(k))|}$.
 The last range is for $k$ small, where we will apply Theorem \ref{thm:smallkold}.

In Appendix \ref{section:stru}, we will complete the proof of proposition \ref{prop:strucmiddle} by some simple computations.

In Appendix \ref{section: disproof}, we will prove that Wood's Conjecture \ref{conj:wood2} is false when $\lambda < 0.553$ (Theorem \ref{thm:woodconj2false}) by checking the number of $k$-cliques in the disjoint union of Tur\'an graphs $T((4t-4)/3,(2t-2)/3)$. 

\section{Analysis of the Peeling Process}\label{sec:container}
The development of the hypergraph container's method has been powerful in answering many long-standing questions. It was developed by Balogh-Morris-Samotij \cite{container1} and Saxton-Thomason \cite{container2}. The idea, which is transferring a general setting into a dense setting,  can be traced back to the classical paper of Kleitman--Winston \cite{KW} on graphs.

Our Key Lemmas in this section are %We will use the container's method (peeling process) for graphs to prove Lemma \ref{lem:reduce2} and 
Lemmas \ref{lem:key1} and \ref{lem:very large counting}, by carefully analyzing the peeling process (container's method \cite{KW}) below. 
The container's method works as follows. Roughly speaking, for each clique $K$ in $G$, we find a way to encode a small number of (ordered) vertices $v_1, \dots, v_{r(K)}$ in $K$, call it $I(K)$. In other words, we gradually peel out vertices from $G$ with vertices in $I(K)$ be the landmarks. We want the total number of encoding $(v_1, \dots, v_{r(K)})$ to be small. Different cliques may have the same encoding and we can group all the cliques $K$ by the different encoding. The vertices in $K \setminus I(K)$ are contained in a ``dense" subgraph of $G$. And we could bound the number of the cliques (of order $k-r(K)$) ( Lemma \ref{lem:largekdense} and Proposition \ref{prop:strucmiddle}).

We now describe the peeling process, which is almost the same procedure as in Fox and the second author \cite{FW} which was heavily motivated by \cite{KW}. However, the analysis of the peeling process in the current paper is much more involved, since we would need to bound the number of cliques for a fixed size $k$. We will elaborate on the difficulties and how we overcome them in the next subsection after the description of the peeling process.

\subsection{Description of the Peeling Process}

 Now we describe how to encode each clique $K$ by some sequence $v_1,\ldots,v_{r(K)}$. To determine the encoding for each clique $K$, we apply the following \textit{peeling process} for $K$.

\begin{peel}
    Firstly, we preorder vertices of $G$. Let $G_0=G$. We delete vertices in $G_0$ one by one until some vertex $v_1 \in K$ has the smallest degree. (We break the tie by the predefined ordering on all the vertices in $G$). In this way, we obtain an induced subgraph $G_1$ that contains $K$ in which $v_1$ has the minimum degree. We repeat this process as follows:   
 \begin{enumerate}
     \item  After picking $v_i$ and thus  obtaining the associated $G_i$, delete from $G_i$ vertex $v_i$ and its non-neighborhood $D_i$. We called this induced subgraph $G'_i$;
     \item Delete vertices in $G'_i$ one by one until some vertex in $K$ has the smallest degree. (We break the tie by the predefined ordering). Let this vertex in $K$ be $v_{i+1}$ and the remaining graph be 
     $G_{i+1}$. 
     We call the set of deleted vertices in this {deleting process} as $Y_i = V(G_i') \setminus V(G_{i+1})$.
 \end{enumerate}

Let $n_i$ be the number of vertices in $G_i$, and also let $d_i$ be the missing degree of $v_i$ in $G_i$, i.e. $d_i=|D_i|$. 
We call the process of finding $v_{i}$ and $G_{i}$ from $G_{i-1}$ the $i$-th step.

We call step $r$ the stopping step and $G_r$ the terminal graph, and let $r(K)=r$ if $r$ is the least positive integer such that 
\begin{enumerate}
\item $n_r \leq t - r$, or
\item $d_r\le \frac{1}{2}(n_{r}+r-t)^{1/2}$,  or
\item $r=|V(K)|$.
\end{enumerate} 

 \end{peel}

For any clique $K$, the peeling process above gives a sequence $v_i$, $G_i$, $D_i$ and $Y_i$. Let the {\it layer} at step $i$ be denoted as $L_i:=D_i\cup Y_i$. Since no more vertices are deleted from the terminal graph $G_r$, for convenience, write $Y_r=\emptyset$.

\begin{figure}[H]
\begin{center}
\includegraphics[width=5cm]{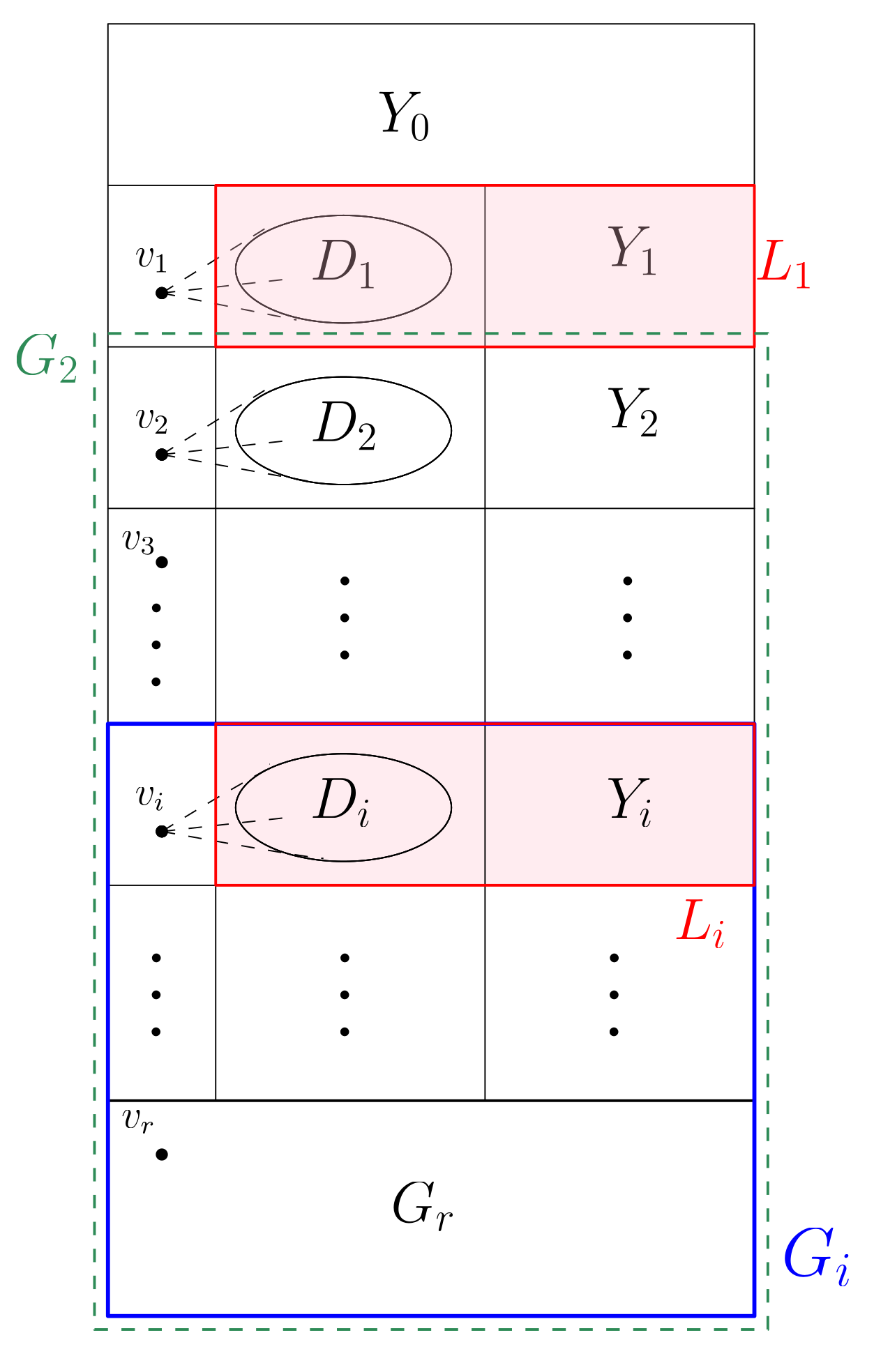}
\end{center}
\caption{Illustration of the notations in the Peeling Process: \footnotesize{ $v_i$ is the vertex with minimum degree in $G_i$, and $D_i$ are the set of non-neighbors of $v_i$ in $G_i$. The set $Y_i$ are the extra vertices deleted until $v_{i+1}$ is the minimum degree vertex. For instance, the red box at the top indicates the layer $L_1$ which is the union of $D_1$ and $Y_1$. 
Notice that $Y_i$ can be an empty set.} }
\label{fig:figure1}
\end{figure}

One reason for applying this peeling process is that, after the first step, we will get a graph  $G_2$ whose size is independent from $n$. The result of Thomason \cite{Th1} implies, as $G$ does not contain a $K_t$-minor, every subgraph of it has a vertex of degree at most $d:=\beta t\sqrt{\ln t}$ when $t$ is sufficiently large. 
{In this paper, without special notice, when we assume $t$ is sufficiently large, we assume $t$ is large enough so that the minimum degree of any $K_t$-minor free graph is at most $\beta t\sqrt{\ln t}$. Since $v_1$ is of minimum degree in $G_1$ and $G_2 \subset N_G(v_1)$, we know \[d_2<n_2 = |G_2| \leq d+1.\] 
 
The stopping condition 2 corresponds to the terminal graph $G_r$ being a ``dense" graph, as the maximum missing degree in $G_r$ is small. The idea from \cite{FW} is that if a graph is dense, then its clique minor size is a simple function in terms of its order and its clique number. 

\begin{lem}[Lemma 2.1, \cite{FW}] \label{lem:densehnumber}
Let $G$ be a graph on $n$ vertices with minimum degree $\delta$ and clique number $\omega$ (the order of the largest clique). Let $\BD=n-\delta-1$, be the maximum missing degree, which is also the maximum degree of the complement of $G$.
We say  $G$ is dense if $n \geq \omega+2\BD^2+2$ or $\BD \leq 1$.
 If $G$ is dense, then the largest $t$ such that $K_t$ is a minor of $G$ is $\lfloor \frac{n+\omega}{2}\rfloor$.
\end{lem}

\begin{definition}
Given a graph $G$, let $\BD(G)$ be its maximum missing degree, and $\omega(G)$ is the order of the largest clique in $G$. We define the following.
\begin{enumerate}
    \item Let $\mathcal{D}$ be the family of all dense graphs, i.e., the set of $G$ such that $|V(G)| \geq \omega(G)+2\BD(G)^2+2$ or $\BD(G) \leq 1$.
    \item Let $\mathcal{G}_s$ be the family of graphs $G$ such that $\lfloor \frac{|V(G)|+\omega(G)}{2} \rfloor \leq s-1$.
    \item Let $\mathcal{H}^s_m$ be the family of graphs $H$ with at most $m$ vertices and clique minor in $H$ has size at most $s$.
\end{enumerate}
    
\end{definition}

Lemma \ref{lem:densehnumber} guarantees that if $G$ is dense and does not have a $K_s$-minor, then $G$ is in $\mathcal{G}_s$. On the other hand, it also showed that if $G$ is dense and is in $\mathcal{G}_{t+1}-\mathcal{G}_t$, then $G$ must contain a $K_t$ minor. Note that there can be graphs that are not dense but also belong to $\mathcal{G}_t$.

\subsection{Analysis of the Peeling Process}\label{subsec:analysis}

In Fox and Wei \cite{FW}, the number of all cliques in $K_t$-minor free graphs is bounded by the product of the error term, which is the number of possible encoding with length at most $r_0=4t^{1/2}{\log_2}^{1/4} t$, and the main term, which is the maximum number of all cliques in graphs in $\mathcal{G}_{t}$. 

When the clique size $k$ is fixed, the error term could be too large. Thus, we need to understand how the parameters given by the peeling process are related to the maximum clique minor size. In Lemma \ref{lem:key1} and Lemma \ref{lem:rs code}, we will show that, in many cases, either the peeling process stops very quickly, and then we have $r(K)$ small and the number of possible encoding for this kind of cliques is small; or the size of the maximum clique minor in the remaining graph $G_r$ is much smaller than $t-r$, resulting in fewer number of such cliques of order $k$. We will also show an improved bound for the number of possible encodings in Lemma \ref{lem:very large counting}.

\begin{definition}
For a given clique $K$ and its peeling sequences, a vertex subset $A\subseteq V(G)$ such that $A\cap K= \emptyset$ and $A \cap V(G_{r(K)})=\emptyset $ is called \emph{an extra branch disk} of $K$ if the induced subgraph $G[A]$ is connected and every vertex in $V(K)\cup V(G_r)$ has at least one neighbor in $A$. 
\end{definition}

To construct a large clique minor, we would like to use the vertices in $K$, together with contracting each branch disk into a single vertex. Thus, we want to find as many disjoint branch disks as possible that are also pairwise adjacent. This motivates us to define the following.

\begin{definition} \label{def: branch vtx set}
  A collection $\mathcal{A}=\{A_1,A_2,\dots,A_s\}$ of disjoint vertex subsets $A_i$ is called \emph{branch vertex set} of $K$, if each $A_i$ is an extra branch disk of $K$, and  for any $1 \leq i < j \leq s$, the two disks $A_i,A_j$ are adjacent, i.e., there exist $x\in A_i,y\in A_j$ such that $xy\in E(G)$. Let $s(K)$ be the maximum size of branch vertex set $\mathcal{A}$ of $K$.
\end{definition}

\begin{claim} \label{claim: branch vtx set}
     Given $K$ and its peeling process which ends in $r = r(K)$ steps. If $G_r$ contains $K_c$ as a minor, then for any branch vertex set $\mathcal{A}$ of $K$, the subgraph induced by $V(G_r)\cup \mathcal{A} \cup \{ v_1,v_2,\dots,v_{r-1}\}$ contains a clique minor of order $r-1+c+|\mathcal{A}|$. 
\end{claim}
Notice that we could choose $c\ge k-r+1$ as $G_r$ contains $k-r+1$ vertices in the clique $K$.
\begin{proof}

By the definition of the peeling process, each $v_j$ where $j\le r-1$ is adjacent to every vertex in $G_r$. Then the claim holds by the connectivity condition in the definitions of extra branch disk and branch vertex set. 
\end{proof}

\begin{definition}
    Suppose  $M$ is a function in terms of $t$ such that $\log M = o(\log t)$. For a fixed clique $K$,
    let $R_M(K)$ be  the number of $i\in [r(K)-1]$ such that $|Y_i| \ge M$. Let $s_M(r,r_l)$ be the minimum value of $s(K)$ among all cliques $K$ with indexes $r(K)=r$ and $R_M(K) = r_l$.
    {If there is no clique $K$ of order $k$ with $r(K)=r$ and $R_M(K)=r_l$, we set $s_M(r,r_l)=+\infty$.}
    %Notice that $r_s\ge 1$ when $M\ge 1$ because $Y_r=\emptyset$.
\end{definition}

The main goal of this section is to prove the following two key lemmas.
Key Lemma 1 will help us control the number of cliques in the terminal graph $G_r$ by stating that $s_M(r,r_l)$ is relatively large when $r$ or $r_l$ is large.

\begin{lem}[Key Lemma 1] \label{lem:key1} 
    
    For large enough $t$, for any fixed $r$, and for any $M=M(t)\ge 1$, we have 
    \[
    \min_{r_l< r} s_M(r,r_l) \ge \frac{r}{3}-7\log_2 t = \frac{r}{3}-O(\log t). 
    \]
    Recall $d = \beta t \sqrt{\ln t}$. Moreover, for any fixed $r$ and $r_l$, and for any $\epsilon\in (0,\frac{1}{6})$ and $M=M(t)\ge 1$, 
    \[ 
    s_M(r,r_l)\ge r_l-1 -7\cdot  \left(\log_{{1}/({1-\epsilon})}d+\left(8r_l\cdot \log_{{1}/({2\epsilon})} {M}\right)/{M}\right).
    \]  
\end{lem}

Key Lemma 2 gives an upper bound on the number of cliques of order $k$ by combining a better error term and the count of $K_{k-r}$ in the terminal graph $G_r$.

\begin{lem}[Key Lemma 2]\label{lem:very large counting} 
Let $r_0 = 4 t^{1/2}{\log_2}^{1/4} t$ and recall $d = \beta t \sqrt{\ln t}$. When $t$ is sufficiently large, for any function $M=M(t)\ge 0$,
the maximum number of cliques of order $k$ in a graph without $K_t$ as a minor is at most 
\[ 
\sum^{\min(r_0, k))}_{r=1} \sum_{\substack{r_l<r: \\ s_M(r,r_l) \le t-k}}  
\left(n\cdot {r-1 \choose r_l}M^{(r-r_l-1)}{r_0 \choose r_l}\left( \frac{d}{r_0}\right) ^{r_l}
% {\alpha t\sqrt{\ln t} \choose r_l}  
\mathcal{N}_{k-r}((\mathcal{G}_{t-r-s_M(r,r_l)+1}\cap \mathcal{D}) \cup \mathcal{H}^{t-r-s_M(r,r_l)}_{t-r}) \right).
\]
\end{lem}

\subsection{Proof of Key Lemma 1 (Lemma \ref{lem:key1})}

In the following paper, when we say two disjoint vertex sets $A$ and $B$ are adjacent, it means there exist vertices $x\in A$ and $y\in B$ such that $xy\in E(G)$. If $\{v\}$ and $B$ adjacent, we simply say $v$ and $B$ are adjacent. The next claim lists some simple facts about the peeling process.

\begin{claim} \label{basic facts}
    The sequence of graphs $G_i,Y_i$ and vertices $v_i \in K$  satisfy the following properties. 
\begin{enumerate}
\item $v_i \in G_i$ and $v_i$ is of minimum degree in $G_i$, and every vertex in $G_i$ has a missing degree at most $d_i$;
\item $G_{i+1}$ does not contain $v_i$ and its non-neighbors in $G$;
\item $G_{i+1}$ contains $K \setminus \{v_1,\ldots,v_i\}$;
\item $G_{i+1}$ is contained in the subgraph of $G$ induced on the vertex set $N_G(v_1) \cap \dots \cap N_G(v_i)$, where $N_G(u)$ denotes the neighborhood of $u$ in $G$. 
\item If $A\subseteq V(G_i)$ and $|A|\ge d_i+1$, then for every $v\in V(G_i)$, either $v\in A$, or $v$ and  $A$ are adjacent. Moreover, if $A\subseteq V(G_i)$ and $|A|\ge 2d_i+1$, then $G[A]$, the subgraph of $G$ induced by $A$, is connected.
\item Suppose $Y_i\neq \emptyset$ and let $y_i\in Y_i$ be the last vertex removed in $Y_i$. Then $y_i$ has no less  non-neighbors in $G_{i+1}$ than $v_{i+1}$, which means $y_i$ has at least $d_{i+1}$ non-neighbors in $G_{i+1}$.
\item Let $y\in Y_i$ be the last vertex removed in $Y_i$. Then the missing degree of  $y$ in $D_i$ is at most $d_i-d_{i+1}$. 
\end{enumerate} 

\end{claim}
\begin{proof}
     Facts 1-4 are clear from the description of the peeling process.
     
     First, we will prove Fact 5. For every vertex $v\in V(G_i)$, by Fact 1, its missing degree in $G_i$ is at most $d_i$, so $v$ must have at least one neighbor in $A$ as $|A|\ge d_i+1$. If  $A\subseteq V(G_i)$ with $|A|\ge 2d_i+1$ and $G[A]$ is disconnected, then one connected component of $G[A]$ has at most $d_i$ vertices. Thus any vertex $u$ in this connected component has at least $|A|-d_i\ge d_i+1$ non-neighbors in $G$. As $A\subseteq V(G_i)$, the missing degree of $u$ in $G_i$ is at least $d_i+1$, which contradicts with Fact 1.

     Suppose Fact 6 is not true. In the subgraph induced by $\{y_i\}\cup V(G_{i+1})$, vertex $y_i$ has less missing degree than $v_{i+1}$.  Thus, $v_{i+1}$ should be deleted before $y_i$ which is a contradiction. 
     Suppose Fact 7 is not true. By Fact 1, the missing degree of $y_i$ in $G_i$ is at most $d_i$ and its missing degree in $G_{i+1}$ is less than $d_{i+1}$ which contradicts Fact 6 in 
 Claim \ref{basic facts}. 
\end{proof}

To prove Lemma \ref{lem:key1} and Lemma \ref{lem:very large counting}, a main step is to show that the number of encodings is small. The following simple claim states that the length of encoding $r(K)$ in the peeling process cannot be too large.

\begin{claim}\label{claim:boundr0}
The length of encoding for each clique $K$ is small. In other words, when $t$ is sufficiently large, 
$r(K) < 4t^{1/2}{\log_2}^{1/4} t$. 
\end{claim}
\begin{proof}
This argument is almost the same as in the paper \cite{FW} and it is mainly due to the fact that before stopping,  the bound on $n_{r} - n_{r+1}$ deduced from the bound of $d_r$ in the second stop condition guarantees that each time $n_r$ drops a lot.  Recall that we set $d=\beta t\sqrt{\ln t}$ where $\beta=0.64$. Recall that the result of Thomason \cite{Th1} implies  $n_2 = |G_2| \leq d+1$ when $t$ is sufficiently large. Let $n_i' = n_i + i -t$.
 \begin{fact} \label{ni' gap}
     For every $i< r$, we have $n'_i-n'_{i+1}>\frac{1}{2} (n'_i)^\frac{1}{2}$, and thus $(n'_i)_i$ is strictly decreasing.
 \end{fact}

 \begin{proof}
By the definition of $r(K)$, before stopping, $d_i> \frac{1}{2} (n'_i)^\frac{1}{2}$ for every $i< r$.  Thus $n'_i-n'_{i+1}=(n_i+i-t)-(n_{i+1}+i+1-t)=n_i-n_{i+1}-1=|L_i|\ge |D_i|=d_i>\frac{1}{2} (n'_i)^\frac{1}{2}$ for every $i<r$.

 Because of the first stopping condition, we have $n_i> t-i$ for every $i<r$. Thus, we have $n'_i> 0$ and $d_i\ge 1$ for every $i<r$.
Thus, $n_i'$ is strictly decreasing before stopping. 
 \end{proof}

For each $0 \leq i \leq 2 \log_2 (d-t)$, let $c_i = (d -t+ 3)/2^i$. 
For any $j \leq r$ with $c_i \geq n_j' \geq  c_{i+1}$, we have $n_j'- n_{j+1}' > \frac{1}{2}(n_j')^{1/2}\geq  \frac{1}{2}c_{i+1}^{1/2}$. Therefore, to drop the $n'_j$ value from $c_i$ to $c_{i+1}$, the number of steps it takes is at most
\[ 1 + (c_i -c_{i+1})/(\frac{1}{2}c_{i+1}^{1/2}) = 1 + 2 c_{i+1}^{1/2} = 1 + 2((d -t+ 3)/2^{i+1})^{1/2}.\]
%values of $j \leq r$ with $d_i \geq n_j-t \geq  d_{i+1}$. 
Note that, when $t$ is sufficiently large, for each $2 \leq j \leq r$, there is some $i \leq 2 \log_2 (d-t)$ with $c_i \geq n_j' \geq  c_{i+1}$ because $n'_j\le n'_2 \le d-t+3$. Thus, 
%\begin{align*}
$r  \leq 1 + \sum_{i=0}^{2 \log_2(d-t)} (1 +  2((d -t+ 3)/2^{i+1})^{1/2})) \\
\leq   1 + 2 \log_2 d + 2(d+3)^{1/2}\sum_{i=0}^{\infty}2^{-(i+1)/2}
<  4t^{1/2}{\log_2}^{1/4} t:=r_0. $
%\end{align*}
The last inequality holds when $t$ is sufficiently large and plugging in $d=\beta t\sqrt{\ln t}$.
\end{proof}

In the following, we will always let $r_0=4t^{1/2}{\log_2}^{1/4} t$. 
Next, we prove the fact that $s_M(r,r_l)$ is relatively large when $r$ is large, which is the first statement in Lemma \ref{lem:key1}.

\begin{lem} \label{lem: half branch}
   {For any $r$ and for any every $k$-clique $K$ with exactly $r$ peeling steps, we have $s(K) \ge  \frac{r-2}{3}-6\log_t d_{2} $ where $d_2$ is determined by $K$. As a consequence, when $t$ is sufficiently large, for any $r$, \[\min_{r_l< r} s_M(r,r_l) \ge \frac{r}{3}-7\log_2 t.\]}
\end{lem} 

\begin{proof}
    For any $k$-clique $K$ with exactly $r$ peeling steps, we have defined $L_i$ and $D_i$ for every $i\le r$ by its peeling process. We show that for every three consecutive layers $L_a,L_{a+1},L_{a+2} $ such that $d_{a+2}\ge \frac{7}{8} d_a$, we can construct an extra branch disk in these three layers.  Notice that, by Fact 5 in Claim \ref{basic facts}, every vertex set in $G_a$ with at least $2d_a+1$ vertices induces a connected subgraph.

    Let $A_a=D_a\cup D_{a+1}\cup D_{a+2}$, so $|A_a|\ge d_a+2\cdot \frac{7}{8} d_a\ge 2d_a+1$. By Fact 5 in Claim \ref{basic facts}, $G[A_a]$ is connected and every vertex in $G_a$ is adjacent to a vertex in $A_a$. Thus, we have $A_a$ is adjacent to every vertex in $V(G_r)$ and $V(K)-\{v_1,v_2,\cdots,v_{a-1}\}$. By Fact 2 in Claim \ref{basic facts}, every vertex in $A_a$ is adjacent to $v_a$ for every $i\in [a-1]$. Thus, $A_a$ is an extra branch disk.

     Set $i_1=2$, recursively define $i_{j+1}$ as the smallest integer such that $d_{i_{j+1}}\le \frac{7}{8}d_{i_j}$. Then we can partition set of all layers except $L_1$ and $L_r$ into brackets of consecutive layers with brackets $P_j=\{ L_{i_j},L_{i_j+1},\dots,L_{i_{j+1}-1} \}$. Thus, there are at most $\log_{\frac{8}{7}}d_2< 6\log_2 d_2$ brackets. Assume there are $l$ brackets and let $d_{i_{l+1}}=r$ for convenience. For any three consecutive layers $L_a,L_{a+1},L_{a+2} $ in the bracket $P_j=\{ L_{i_j},L_{i_j+1},\dots,L_{i_{j+1}-1} \}$, we can construct a branch vertex $A_a$, so we can construct $\lfloor \frac{i_{j+1}-i_j}{3} \rfloor$ branch vertices $A_{i_j}, A_{i_j+3},A_{i_j+6},\dots$ in this bracket $P_j$. In total, we construct at least
    \[
    \sum_{j=1}^{l} \lfloor \frac{i_{j+1}-i_j}{3} \rfloor \ge \sum_{j=1}^{l} \left( \frac{i_{j+1}-i_j}{3} -1 \right) \ge
    \left( \sum_{j=1}^{l} \frac{i_{j+1}-i_j}{3} \right) - l\ge 
    \frac{r-2}{3}-6\log_t d_{2}
    \]
disjoint branch vertices. These extra branch disks are pairwise adjacent which means they form a branch vertex set. When $t$ is sufficiently large, we have $d_2\le d \le \beta t\sqrt{\ln t}$ and have $s(K) \ge \frac{r}{3}-7\log_2 t$ for every $k$-clique $K$.

\end{proof}

\begin{rem}
    With more effort, we can show that, for every $k$-clique $K$ with exactly $r$ peeling steps,  $s(K) \ge \frac{r}{2}-O(\log t)$ when $t$ is sufficiently large. However, Lemma \ref{lem: half branch} is good enough to prove the main result Theorem \ref{thm: summary} for very large $k$, which is the Theorem \ref{thm:klarge0}. 
\end{rem}

When $r_l$ is large, i.e., there are many layers with large $Y_i$, we can expect to find more branch vertices. Now we will prove the second statement of Lemma \ref{lem:key1} in Lemma \ref{lem:nonempty Y}, which is more technical than the proof of Lemma \ref{lem: half branch}.

The rough idea is as follows. Suppose we have a sequence of layers whose $D_i$ do not differ by much in sizes, then we will first try to construct the extra branch disk in the topmost layer with non-empty $Y_i$. Recall that $y_i$ is the last removed vertex in $Y_i$. If the subgraph induced by $D_i\cup \{y_i\}$ is connected, by Fact 5 in Claim \ref{basic facts}, it can be contracted as an extra branch disk that is adjacent to every vertex in $G_i$ as $|V(D_i\cup \{y_i\})|= d_i+1$. If not, we will try to use some vertices in lower layers to connect the different connected components of $D_i \cup \{y_i\}$, and make all these vertices an extra branch disk (Claim \ref{claim: branch vtx2}). We will try to construct the other extra branch disks greedily. Suppose in layer $j$, the set $U$ is the set of vertices that have not been used. Then Claim \ref{claim: branch vtx2} 
 will show either $U$ itself could be an extra branch disk, or we could add to $U$ a small set of vertices from lower layers so that this set together with $U$ is a valid extra branch disk. We will show that by first preprocessing the layers properly, this greedy construction process will work for most of the layers (Claim \ref{claim: branch process}).

\begin{claim} \label{claim: branch vtx2}
    Fix $\epsilon\in (0,\frac{1}{6})$. Suppose $L_i$ is a layer with nonempty $Y_i$. Let $y=y_i$ be the last removed vertex in $ Y_i$ in the peeling process. For every $U\subseteq D_i$ such that $|U|\ge 2\epsilon d_i$ and $d_{i+1}\ge (1-\epsilon)d_i$,  and for every $X\subseteq G_{i+1}$ such that $|X|\ge 3.5 d_i$, there exists $W\subseteq X$ that the subgraph induced by $U\cup W\cup \{y\}$ is connected, and $d_i\le |U\cup W|\le d_i + %\textcolor{red}
    {2\log_\frac{1}{2\epsilon} d_i}$.
\end{claim} 

\begin{proof}
     By Fact 7 in Claim \ref{basic facts}, the missing degree of $y$ in $D_i$ is at most $d_i-d_{i+1}\le \epsilon d_i$, so  $|N(y)\cap U|\ge |U| - \epsilon d_i \geq  2\epsilon d_i -\epsilon d_i=\epsilon d_i$ which means $N(y)\cap U$ is non-empty. Let $L$ be the connected component in $L_i$ containing $y$ and $ N(y)\cap U$ (clearly $y$ is adjacent to every vertex in  $N(y)\cap U$). Let $R=U-L$ 
    %{\color{red}Fan: explain.}, 
    and $R_1,R_2,\dots,R_l$ be the connected components in the graph induced by $R$. Because $R$ and $y$ are not adjacent, $|R|$ is upper bounded by the missing degree of $y$ in $D_i$. In other words, \begin{equation}
        |R|\le d_i-d_{i+1} %\le d_i-(1-\epsilon)d_{i_0} 
    \le d_i-(1-\epsilon)d_{i}=\epsilon d_i. \label{eq:sizeR}
    \end{equation} 
     Let $X'=N(y)\cap X$ and $O=D_i-U$. An illustration is shown in Figure \ref{fig:figure2}.

\begin{figure}[htb]
\begin{center}
\includegraphics[width=8cm]{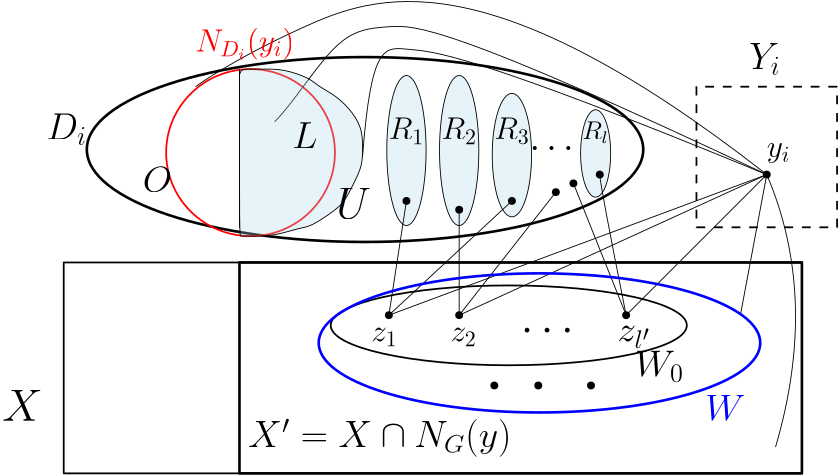}
\end{center}
\caption{Construct an extra branch disk for layer $L_i$}
\label{fig:figure2}
\end{figure}

    Our goal is to find a vertex set $W$ in $X'$ to connect $L$ and the $l$ connected components of $R$. Since $y$ has at most $d_i$ missing edges in $G_i$, we have $|X'|\ge |X| -d_i\ge 2.5d_i$.
    To find the proper $W$, we will first try to find a set $W_0\subseteq X'$ such that $\{y\}\cup U \cup W_0$ is connected and $|W_0|\le |O|+\log_\frac{1}{2\epsilon} d_i$. If $|W_0|\ge |O|$, then we let $W=W_0$; if $|W_0|\le |O|$, we will add $|O|-|W_0|$ vertices from $X'\setminus W_0$ to $W_0$ to construct $W$. The proof falls into the following two cases:

\paragraph{Case 1: $|O|\ge l$.} Recall $l$ is the number of connected components in the graph induced by $R$. \\
We will construct $W_0$ by picking one vertex $z_i$ in $X'$ for each $R_i$ where $i\in[l]$. Here the $z_i$'s do not need to be distinct.
For any vertex $v\in R_i\subseteq G_i$, it has at most $d_i$ non-neighbors in $G_i$. Thus, this vertex $v$ has at least one neighbor in $X'$ as $X'\subseteq G_i$ and $|X'|\ge 2.5d_i$. Then we select one of these neighbors arbitrarily as $z_i$. Let $W_0=\{z_1,z_2,\dots,z_l\}$.

Because $|W_0|\le l \le |O|$, we have $|U\cup W_0|\le |U|+|O|=|D_i|=d_i$. To construct a proper $W$, we need to add $|O|-|W_0|$ vertices from $X'\setminus W_0$ to $W_0=\{z_1,z_2,\dots,z_l\}$ to form the set $W$. We can do this addition because $|X'\setminus W_0|\ge |X'|-|O|\ge  2.5d_i  - d_i > d_i \ge |O|-|W_0|$. 
Thus we have that the subgraph induced by $U\cup W\cup \{y\}$ is connected as every vertex in $W\subset X'$ is adjacent to $y$, and  $z_i \in W_0$ is connected to $y$ and the connected component $R_i$. The size satisfies $|U\cup W|= (d_i -|O|)+|W_0|+(|O|-|W_0|)=d_i$.

\paragraph{Case 2: $|O|\le l$.}
Because $l$ is the number of connected components in the graph induced by $R$, we have $|R|\ge l$. Thus, $|O|\le l \le |R|\le 2 \epsilon d_i$ by (\ref{eq:sizeR}). 
For every $j\in[l]$, for any vertex $v\in R_j$, $v$ has at least $|L|$ nonneighbors in $L_i$ because $R_j$ is disconnected from $L$.  So $v$ has at most $d_i-|L|= |R|+|O|\le 4\epsilon d_i$ non-neighbors in $X'$ as we just showed $|O|\le  |R|\le 2 \epsilon d_i$. 

Let $l'=\lfloor 2\log_\frac{1}{2\epsilon} d_i \rfloor$. We now find a set $W_0 \subset X'$ with size $l'$ such that $\{y\}\cup U \cup W_0$ is connected. To do this, we pick $l'$ vertices independently uniformly at random from $X'$, and let them be the set $W_0$. 
The event $|W_0|<l'$ happens when some vertex was selected more than once. By a union bound, 
$
\text{Pr}(|W_0| < l') \leq \binom{l'}{2}\cdot \frac{|X'|}{|X'|^2}<\frac{(l')^2}{2|X'|}\le \frac{(l')^2}{5d_i}.
$

When $|W_0|=l'$, for every $j\in[l]$, the probability that there is no edge between $W_0$ and $R_j$ is at most
$
\left( \frac{4\epsilon d_i}{|X'|}\right)^{l'}\le
\left( \frac{4\epsilon d_i}{2.5d_i}\right)^{l'}=
\left( \frac{8\epsilon }{5}\right)^{l'} < (2\epsilon)^{l'}.
$ 
By a union bound, the probability that $|W_0|<l' $ or $|W_0|=l'$ but 
there exists a $j\in [l]$ such that $W_0$ and $R_j$ is not adjacent is at most
\[
\frac{(l')^2}{5d_i}+l\cdot (2\epsilon)^{l'}\le 
\frac{(l')^2}{5d_i}+2\epsilon d_i \cdot (2\epsilon)^{l'}\le \frac{(l')^2}{5d_i}+2\epsilon d_i \cdot (2\epsilon)^{(2\log_\frac{1}{2\epsilon} d_i) -1}  =\frac{(2\log_\frac{1}{2\epsilon} d_i)^2}{5d_i}+\frac{1}{d_i} < 1
\]
The second inequality holds because  $l'=\lfloor 2\log_\frac{1}{2\epsilon} d_i \rfloor \ge 2\log_\frac{1}{2\epsilon} d_i -1$. The last inequality is true since $1/(2\epsilon) \ge 3$ as $\epsilon\in (0,\frac{1}{6}]$.  
The union bound showed that there is a subset $W_0 \subseteq X'$ with $l'=\lfloor 2\log_\frac{1}{2\epsilon} d_i \rfloor$ vertices such that every $R_j$ in $R$ has at least one neighbor in $W_0$. Thus, together with the fact that every vertex in $X'$ is adjacent to $y$, we have that $\{y\}\cup U \cup W_0$ is connected.

We now construct a set $W$ with the desired size. If $|O|\le l'$, let $W=W_0$; if $|O|> l'$, then add $|O|-l'$ vertices from $X'$ to $W_0$, and let this new set be $W$. We can do this because 
$
|X'\setminus W_0|\ge |X'|-|O|\ge  2.5d_i  - d_i >  d_i=|D_i| \geq |O| > |O|-l'.
$
The first inequality is because $|W_0|= l'\leq  |O|$, and the second inequality is because $|X'|\ge 2.5d_i$ and $O\subseteq D_i$. 
By the definition of $X'$, every vertex in $X'$ is adjacent to $y$. By the fact that $W\subset X'$ and that we have just shown $\{y\} \cup U \cup W_0$ is connected, we have that $U\cup W\cup \{y\}$ is connected. For the size of $U\cup W$, we have $|U\cup W|=d_i-|O|+\max\{|O|,l'\}$. Therefore $d_i\le |U\cup W|\le d_i+l' \le  d_i+2\log_{1/2\epsilon} d_i$.
\end{proof}

In the next claim, we apply Claim \ref{claim: branch vtx2} to consecutive layers to construct extra branch disks.  To apply Claim \ref{claim: branch vtx2}, we need to cut all layers into brackets such that $d_i/d_j$ is close to $1$ for any two layers $L_i,L_j$ in the same bracket. 
For a  fixed $\epsilon$, we set $p_j=(1-\epsilon)^jd_2$ for every $j\ge 0$, and then let $P_j$ be the set of layers $L_i$ such that $d_i\in (p_j,p_{j-1}]$. 
Because $d_i$ is non-increasing as $i$ increases, we partition the set of all layers (except $L_1$ and $L_r$) into brackets of consecutive layers $P_j=\{ L_{i_j},L_{i_j+1},\dots,L_{i_{j+1}-1} \}$. For any $j$, for any $L_a,L_b\in P_j$, we have 
\[
d_a> (1-\epsilon)^j d_2\ge (1-\epsilon)(1-\epsilon)^{j-1}d_2\ge (1-\epsilon)d_b. 
\]
Also, there are at most $\log_{\frac{1}{1-\epsilon}}d_2$ brackets. We will try to create branch vertices from vertices in the same bracket.  
\begin{claim} \label{claim: branch process}
  For any fixed $\epsilon\in (0,\frac{1}{6}]$, for any $T$ layers  $\{L_{a},L_{a+1},\dots,L_{a+T-1}\}$ in the same bracket where $T\le \frac{d_a}{4\log_{\frac{1}{2\epsilon}}d_a}$ and $Y_{a+l}\neq \emptyset$ for every $l\in [0,T-8]$, 
  we can create $T-7$ disjoint extra branch disks $\{A_a,A_{a+1},\dots, A_{a+T-8}\}$ which are disjoint with $V(K)$, such that $A_i\subseteq V(G_i)$ and $ 1+d_i \le |A_i| \le 1+d_i+\log_{\frac{1}{2\epsilon}}d_i$, and $D_i\subseteq A_a\cup A_{a+1}\cup \cdots \cup A_{i}$ for every $i\in[a,a+T-8]$.   
\end{claim}

{
\begin{rem}
    The proof does not require these $T$ layers to be consecutive in the original peeling process. For any $T$ layers in the same bracket, if we relabel indices of them by $\{a,a+1,\cdots,a+T-1\}$ based on their original order, then the same result follows. 
\end{rem}}

\begin{proof}

    For $T\le 7$, this claim is trivially true. Assume $T\ge 8$. Let $y_i$ be the last removed vertex in $Y_i$ for each $i$ and $Y$ be the set of all $y_i$'s for $a \leq i \leq a+T-1$. We will create $T-7$ branch vertices $\{A_a,A_{a+1},\dots, A_{a+T-8}\}$ recursively below:
\begin{enumerate}
    \item Initial Step: Let $U=D_{a}$ and $X=L_{a+1}\cup \dots \cup L_{a+T-1}-Y$.  
    Thus, we have $|U|=d_a> 2\epsilon d_a$ and $|X|\ge (T-1)|D_{a+T-1}|\ge (T-1)(1-\epsilon)d_a\ge 3.5d_a$, and $d_{a+1}\ge (1-\epsilon)d_a$ because $L_a$ and $L_{a+1}$ are in the same bracket. Then, by Claim \ref{claim: branch vtx2}, we can find $W\subseteq X$ such that $A_a=U\cup W \cup \{y_a \} \subseteq V(G_a)$ and $d_a\le |U\cup W|\le d_a+\log_{\frac{1}{2\epsilon}}d_a$ and $D_a=U\subseteq A_a$.
    \item $l$-th Step: Suppose we have found the desired $A_a, A_{a+1}, \dots, A_{a+l-1}$ for some $l \geq 1$. 
    Let $U$ be the unused vertices in $D_{a+l}$, 
    i.e. $U=D_{a+l}- (A_a\cup\dots\cup A_{a+l-1})$. Let $X$ be the unused vertices in the lower layers in the same bracket excluding the vertices in $Y$, i.e., $X=(L_{a+l+1}\cup \dots \cup L_{a+T-1})-(A_a\cup\dots\cup A_{a+l-1})-Y$. By definition, $X\subseteq G_{a+l}$.
    
    In  Fact \ref{claim:condition} below, we will show the conditions of Claim \ref{claim: branch vtx2} hold for the $U$ and $X$ defined in the $l$-th step. Then Claim \ref{claim: branch vtx2} will guarantee a subset $W\subseteq X$ such that $A_{a+l}=U\cup W \cup \{y_{a+l} \}\subseteq V(G_{a+l})$ and $d_{a+l}\le |U\cup W|\le d_{a+l}+\log_{\frac{1}{2\epsilon}}d_{a+l}$ and $D_{a+l}\subseteq U \cup (A_a\cup\dots\cup A_{a+l-1}) \subseteq A_{a+l}\cup (A_a\cup\dots\cup A_{a+l-1})$. We call $A_{a+l}$ the branch vertex constructed in layer $L_{a+l}$. 
\end{enumerate}

\begin{fact} \label{claim:condition}
    For every $l\in [0,T-8]$, in the $l$-th step defined in the proof of Claim \ref{claim: branch process},  we have $d_{a+l+1}\ge (1-\epsilon) d_{a+l}$, and $|U|\ge 2\epsilon d_{a+l}$ and $|X|\ge 3.5d_{a+l}$.
\end{fact}
\begin{proof}
The case $l=0$ is already proved in the Initial Step in Claim \ref{claim: branch process}. Also, the condition $d_{a+l+1}\ge (1-\epsilon) d_{a+l}$ is always true because $L_{a+l}$ and $L_{a+l+1}$ are in the same bracket. 

Suppose the Fact \ref{claim:condition} works for the first $l-1$ steps for some $l \geq 2$. We now show the $l$-th step also works. Let $A'_i=A_i-\{y_i\}$ and $L'_i=L_i-\{y_i\}$. By the inductive hypothesis, $|A'_i|\le d_i + 2\log_{\frac{1}{2\epsilon}}d_i \le d_i +2 \log_{\frac{1}{2\epsilon}}d_a$ for every $i\in [a, a+l-1]$. Thus
we have $|A'_a\cup\cdots\cup A'_{a+l-1}|\le (\sum_{i=a}^{a+l-1}d_i)+l\cdot 2\log_{\frac{1}{2\epsilon}}d_a$.

By the inductive hypothesis, $D_i\subseteq A_a\cup \cdots \cup A_{i}$ for every $i\in [a,a+l-1]$.
Because $D_i$ is disjoint from any $y_j$ for any $i,j\le r$, we have $D_i\subseteq A'_a\cup \cdots \cup A'_{i}$ for every $i\in [a,a+l-1]$. Thus, 
\begin{equation}
    D_a\cup D_{a+1}\cup\dots \cup D_{a+l-1}\subseteq A'_a\cup\dots\cup A'_{a+l-1}. \label{eq:inclusion}
\end{equation} and we know $|D_a\cup D_{a+1}\cup\dots \cup D_{a+l-1}|=\sum_{i=a}^{a+l-1}d_i$. 

We now upper bound  the number of used vertex in $G_{a+l}$, which is $|V(G_{a+l})\cap (A'_a\cup\dots\cup A'_{a+l-1})|$. 
Because $D_a\cup D_{a+1}\cup\dots \cup D_{a+l-1}$ is disjoint from $G_{a+l}$, together with (\ref{eq:inclusion}), we have that  
 \begin{align}
 & |V(G_{a+l})\cap (A'_a\cup\dots\cup A'_{a+l-1})| \\
 \leq &|A'_a\cup\dots\cup A'_{a+l-1}| - |D_a\cup D_{a+1}\cup\dots \cup D_{a+l-1}| =(\sum_{i=a}^{a+l-1}d_i+l\cdot 2\log_{\frac{1}{2\epsilon}}d_a) -\sum_{i=a}^{a+l-1}d_i  \\    
\le & T\cdot 2\log_{\frac{1}{2\epsilon}}d_a\le \frac{d_a}{4\log_{\frac{1}{2\epsilon}}d_a}\cdot 2\log_{\frac{1}{2\epsilon}}d_a=d_a/2. \label{eq:unused}
 \end{align} 

 We can now bound $|U|$.
In the $l$-th step, the set of unused vertices in $D_{a+l}$, i.e., the set $U = D_{a+l} \setminus (A_a\cup\dots\cup A_{a+l-1})$, satisfies $|U| = |D_{a+l}| - |D_{a+l}\cap (A_a\cup\dots\cup A_{a+l-1})|=|D_{a+l}| - |D_{a+l}\cap (A'_a\cup\dots\cup A'_{a+l-1})|$ is at least
\[
|D_{a+l}|-|V(G_{a+l})\cap (A'_a\cup\dots\cup A'_{a+l-1})| \ge d_{a+l}-d_a/2\ge (1-\epsilon)d_a-d_a/2\ge 2\epsilon d_a \ge 2\epsilon d_{a+l}. 
\] where the first inequality is by (\ref{eq:unused})and the third inequality is true because $\epsilon \le \frac{1}{6}$.

We now bound $|X|$. Because $L_{a+l+1}\cup \dots \cup L_{a+T-1}\subseteq V(G_{a+l+1})\subseteq V(G_{a+l})$, then we have 
\begin{equation}
   |(L'_{a+l+1}\cup \dots \cup L'_{a+T-1})\cap (A'_a\cup\dots\cup A'_{a+l-1})|\le |V(G_{a+l})\cap (A'_a\cup\dots\cup A'_{a+l-1})| \le d_a/2.   \label{eq:1}
\end{equation}
We can now show $|X|\ge 3.5d_a$ as
\begin{align*}
|X|=&|(L_{a+l+1}\cup \dots \cup L_{a+T-1})-(A_a\cup\dots\cup A_{a+l-1})-Y|\\
    =&|(L'_{a+l+1}\cup \dots \cup L'_{a+T-1})-(A'_a\cup\dots\cup A'_{a+l-1})|\\
    =&|L'_{a+l+1}\cup \dots \cup L'_{a+T-1}|-|(L'_{a+l+1}\cup \dots \cup L'_{a+T-1})\cap (A'_a\cup\dots\cup A'_{a+l-1})|\\
    \ge &(T-l)(1-\epsilon)d_a -d_a/2
    \ge 4d_a-d_a/2\ge 3.5d_a \ge 3.5 d_{a+l}.
\end{align*}
where the first inequality is by (\ref{eq:1}) and the second inequality is by the fact that $l \leq T-8$.
\end{proof}

Next, we show each $A_i$ constructed by this process is an extra branch vertex. The induced subgraph of $A_i$ is connected which is guaranteed by Claim \ref{claim: branch vtx2}. For $v_1,v_2,\cdots,v_{i-1}\in V(K)$, they are all adjacent to $y_i\in V(G_i)$ by Fact 2 in Claim \ref{basic facts}. 
Also, $y_i$ and $v_i$ are adjacent by definition of $Y_i$ in the peeling process.   The rest of vertices of $K$ and $V(G_r)$ are contained in $V(G_i)$. Because $|A_i|\ge d_i+1$, by Fact 5 in Claim \ref{basic facts}, we can show the rest of the vertices in $K$ and $V(G_r)$ are all adjacent to $A_i$. Thus, $A_i$ is an extra branch disk.
We have completed the proof of Claim \ref{claim: branch process} and will show that $A_i$ and $A_j$ are adjacent later.
\end{proof}

\begin{lem} \label{lem:nonempty Y}
   Let $t$ be sufficiently large. For any fixed $r$ and $r_l$, and for any $\epsilon\in (0,{\frac{1}{6}}]$ and $M=M(t)\ge 1$, we have  
    \[ 
    s_M(r,r_l)\ge r_l-1 -7\cdot  (\log_{{1}/({1-\epsilon})}d+{8r_l\cdot \log_{{1}/({2\epsilon})} M}/{M}).
    \]
\end{lem}
\begin{proof}
    For any $k$-clique $K$ with indices $r(K)=r$ and $R_M(K)=r_l$, we have defined $v_i$, $L_i$, $D_i$, $d_i$ and $Y_i$ for every $i\le r$ by its peeling process. To prove this lemma, we will try to construct a branch vertex for almost every layer $L_i$ with $|Y_i| > M$ by combining the last removed vertex $y_i$ in $Y_i$ and its neighbors in $D_i$ with a small set of vertices in the lower layers. %{\color{orange}Fan: this is not true because the branch vertex will use vertices in later layers. Also say something about the goal: we need each vertex set to be connected, and adjacent to every vertex in $G_i$.} 
    
    We first consider the layers $L_i$ with $|Y_i|\ge \max(M, 2d_i+1)$. For each such layer $L_i$, we claim $Y_i$ could be a branch disk. This is because by Fact 5 in Claim \ref{basic facts}, the induced subgraph $G[Y_i]$ is connected and adjacent to every vertex in $V(G_i)$.
 The vertices $v_1,v_2,\cdots,v_{i-1}\in V(K)$ are all adjacent to $Y_i \subset V(G_i)$ by Fact 2 in Claim \ref{basic facts}. Also, $y_i$ and $v_i$ are adjacent by the definition of $Y_i$ in the peeling process. 
 Suppose there are $r_l-r'_l$ layers $L_i$ with $|Y_i|\ge \max(M, 2d_i+1)$ and thus we have already constructed $r_l - r_l'$ branch vertices only using vertices in the layers $L_i$ where $|Y_i|\ge \max(M, 2d_i+1)$. 

  We now only consider the layers with $M \leq |Y_i| < 2d_i+1$.  
  For convenience, we remove the  layers $L_i$ with $|Y_i|< M$ or $|Y_i|\ge \max(M, 2d_i+1)$.
Say there are $r'_l \geq 0$ layers left. To prove the lemma, it suffices to prove that we can construct $r'_l-1 -7\cdot  (\log_{\frac{1}{1-\epsilon}}d_2+({8r_l\cdot \log_{\frac{1}{2\epsilon}} M})/{M})$ branch disks in the remaining $r_l'$ layers. 
    
    For our convenience, relabel the indices $i$ of the remaining layers $L_i$'s in order, and thus rename the index $i$ in the corresponding $v_i$, $D_i$, $d_i$, and $Y_i$'s. We then have layers $L_1, \dots, L_{r'_l}$. Because every layers $L_i$ with $|Y_i|< M$ or $|Y_i|\ge 2d_i$ was removed, we have $d_i\ge \frac{1}{2}|Y_i|\ge \frac{M}{2}$ for every $i\le r'_l$.

Let $T_0=\frac{d_{r'_l}}{4\log_{\frac{1}{2\epsilon}}d_{r'_l}}$, so $T_0\ge \frac{M}{8\log_{\frac{1}{2\epsilon}}M}$. 
For each bracket $P_j$, we partition layers in this bracket into intervals of consecutive layers where each interval has  $T$ layers except possibly the last interval which may have fewer than $T_0$ layers. %Then we have some intervals partitioned $P_j$ and 
We call these intervals the processing intervals in $P_j$.

For any processing interval, suppose $L_a$ is the first layer in this interval. Then $\frac{d_a}{4\log_{\frac{1}{2\epsilon}}d_a} \geq \frac{d_{r'_l}}{4\log_{\frac{1}{2\epsilon}}d_{r'_l}} = T_0$. 
Thus, we can apply Claim \ref{claim: branch process} for this interval, and then we can construct an extra branch disk $A_i$ for every layer $L_i$ in this interval except the last 7 layers. 
Now we are ready to complete the proof of Lemma \ref{lem:nonempty Y}.
Recall that there are at most $\log_{\frac{1}{1-\epsilon}}d_2$ brackets $P_j$, so there are at most $\log_{\frac{1}{1-\epsilon}}d_2$ processing interval with fewer than $T_0$ layers. Furthermore, since there are $r'_l$ layers in total, there are at most ${r'_l}/{T_0}$ processing intervals with $T_0$ layers.
Because $L_1$ and $L_r$ are not in any brackets and $Y_r=\emptyset$, the number of branch disks we could construct is at least 
\[
r'_l-1-7\cdot  (\log_{\frac{1}{1-\epsilon}}d_2+{r_l'}/{T_0})\le r'_l-1-7\cdot  \left(\log_{\frac{1}{1-\epsilon}}d_2+\left({8r_l'\cdot \log_{\frac{1}{2\epsilon}} M}\right)/{M}\right).
\] 
 
 Let $\mathcal{A}$ be the set of all these extra branch disks $A_i$ together with all extra branch disks $Y_j$ for the layer with $|Y_j|\ge \max\{M, 2d_j+1\}$ (before removing the layers $L_i$ with $|Y_i| < M$ or $|Y_i| \geq \max\{M, 2d_j+1\}$). 
We now show that $\mathcal{A}$ is a branch vertex set of $K$. The condition we need to check is that any two extra branch disks $A,B\in \mathcal{A}$ are adjacent. 
Assume $A$ is the extra branch disk for layer $L_i$ and $B$ is the extra branch disk for layer $L_j$ in the original peeling process (without removing layers). Without loss of generality, assume $i<j$.  Because $|A|\ge d_i+1$ and $B\subseteq V(G_j)\subseteq V(G_i)$. By Claim \ref{basic facts} Fact 5, every vertex in $B$ is adjacent to some vertices in $A$, so $A$ and $B$ are adjacent.  

Thus, by definition of $s_M(r,r_l)$, we have $s_M(r,r_l)\ge |\mathcal{A}|$, and 
$
 |\mathcal{A}|\ge (r_l-r'_l)+  (r'_l-1-7\cdot  (\log_{\frac{1}{1-\epsilon}}d_2+ \left({8r_l'\cdot \log_{\frac{1}{2\epsilon}} M}\right)/{M}) ).
$
When $t$ is sufficiently large, we have $d_2\le d$.
Thus, we proved the following lower bound 
$
 s_M(r,r_l)\ge |\mathcal{A}| \ge  r_l-1 -7\cdot  \left(\log_{\frac{1}{1-\epsilon}}d+\left({8r_l'\cdot \log_{\frac{1}{2\epsilon}} M}\right)/{M}\right). 
$ 
\end{proof}

\subsection{Proof of Key Lemma 2 (Lemma \ref{lem:very large counting})}

To prove Lemma \ref{lem:very large counting}, a main step is to show that the number of encodings is small. In \cite{FW}, a crude bound $\binom{\beta t \sqrt{\ln t}}{r_0}$ was sufficient. However, this error bound could be too large if we want to count the cliques of a fixed size $k$. In the next lemma, we provide an improved bound on the number of encoding of $k$-cliques $K$ with indices $r(K)=r$ and $R_M(K)=r_l$. 
Recall that the length $r(K)$ of encoding of any clique $K$ is at most $r_0=4t^{1/2}{\log_2}^{1/4} t$.

\begin{lem}\label{lem:rs code}
    For fixed $r$, $r_l$ and function $M = M(t)$, the number of possible encoding of k-cliques $K$ with $r(K)=r$ and $R_M(K)=r_l$ is at most 
    $n{r-1 \choose r_l}M^{r-r_l-1}{r_0 \choose r_l}(\frac{\beta t\sqrt{\ln t}}{r_0})^{r_l}$
\end{lem}

\begin{proof}
For a given clique $K$, we separate the steps $1 \leq i\leq r(K) = r$ depending on whether $|Y_i|$ is large or not. To be more precise, let $L(K)\cup S(K)=[r]$ be the partition of $[r]$ such that $L(K)=\{i\in [r] \vert |Y_i|\ge M\}$ and $S(K)=\{i\in [r]| |Y_i|< M\}$. So $|S(K)| = r-r_l$. 
For any fixed subset $L\subset [r]$, let $C(L)$ be the set of all possible encoding of cliques $K$ such that $L(K)=L, S(K)=[r]\setminus L$. 

We first bound the size of $C(L)$ for any given $L \subset [r]$. The first vertex $v_1$ has $n$ choices. Once $v_1$ is fixed, all the rest of the vertices will be picked from $N(v_1)$, which has order at most $d$.

\begin{claim} \label{ni determine vi}
After picking $v_1$, the vertices $v_2, \dots, v_i$ are uniquely determined by $n_2, \dots, n_i$. 
\end{claim}
\begin{proof}
We will prove by induction that both $v_i$ and $G_i$ are uniquely determined by $n_2, \dots, n_i$ after picking $v_1$. 
After picking $v_1$, we have the unique $G_1$ where $v_1$ has the minimum degree. This is because $G_1$ is obtained from $G$ by removing vertices one at a time degrees smaller than $v_1$. Then from $G_1$, we remove $v_1$ and the non-neighbors of $v_1$, obtaining $G_1'$. Thus $G_1'$ is also uniquely determined by $v_1$. In $G_1'$, we sequentially remove vertices of degrees smaller than $v_2$ (breaking ties by some predetermined order) until in $G_2$, vertex $v_2$ has the minimum degree. So by knowing how many vertices we delete from $G_1'$ to get $G_2$, we will know $v_2$. The number of vertices we delete in this step is $|G_1'| - |G_2|$. However, $|G_1| - |G_1'|$ is also uniquely determined by $v_1$ as shown before.  Thus we know $v_2, G_2$ are uniquely determined by $|G_2| - |G_1|$. The base case holds.  
 
Suppose we have found $v_2, \dots, v_i$ where $v_j, G_j$ for $j\leq i$ are uniquely determined by $n_2, \dots, n_j$. We have determined a graph $G_i$ where $v_i$ is of minimum degree. Similarly, the graph $G_{i+1}$ is 
the induced subgraph of $G_i$ after removing $v_i$ and the non-neighbors of $v_i$, and then we delete from $G_{i+1}$ other vertices till $v_{i+1}$ is the minimum degree (after breaking the tie by some predetermined order). By a similar argument as before, $G_{i+1}$ and $v_{i+1}$ are uniquely determined by knowing how many vertices are deleted from $G_i$ given $G_i$ and $v_i$. Thus the inductive hypothesis holds. 
\end{proof}

Recall that we have shown the sequence $n'_i$ is strictly decreasing before stopping in Fact \ref{ni' gap}. The claim above has the following corollary. 
\begin{cor} \label{cor: n' fix v}
After picking $v_1$, the vertices $v_2, \dots, v_i$ are uniquely determined by $n_2', \dots, n_i'$ where $n_i' = n_i + i -t$ and $n_i'$ are strictly decreasing.
\end{cor}

For any $L$ and $S$, assume we have already selected $v_1,v_2,\cdots,v_{i-1}$, which are the first $i-1$ vertices in some encoding in $C(L)$. Then  $G_1,G_2,\cdots,G_{i-1}$ and
 $n'_1,n'_2,\cdots,n'_{i-1}$ are also determined. To select $v_i$ such that $v_1,v_2,\cdots,v_{i}$ are the first $i$ vertices in some encoding in $C(L)$, it suffices to select the number $n'_{i}$ by Corollary \ref{cor: n' fix v}. If $i-1\in S$, then $n'_{i-1}-n'_i$ is not too large by the definition of $S$. 
Because $n_i'$ are strictly decreasing before stopping, we will see that there are not too many choices for $n'_i$ if $i-1\in S$. Thus, we define $S'= \{i\in [r]| i-1 \in S \}$ and $L'=\{i\in [r]| i-1 \in L \}$, and we will bound the number of choices of $n_i$ in $L'$ and $S'$ separately. Then we have $L'\cup S'=[r]\setminus \{1\}$ is a partition of $[r]\setminus\{1\}$. In addition, $|L'|=|L|$ and $|S'|=|S|-1$ because $r\in S$ by the definition of $S$.

Next, we bound the number of possible subsequences of $n'_{i_1},n'_{i_2},\cdots,n'_{i_{|L'|}}$ where $i_j\in L'$. 
Because of the first stopping condition, for every $i\neq r$, we have $n_i> t-r$. Thus, for every $i\neq r$, $n'_i\in [1, d+1-t]$. Also, $n'_r=n_r+r-t\ge (k-r+1)+r-t=k+1-t\ge -t$ because $G_r$ contains $k-r+1$ vertices in $V(K)$ and $k\ge 2$. % Also, $n'_r\ge 1+2-t=-t+3$. 

We partition the interval $(0, d+1-t]$ into intervals $I_i= (b_{i+1}, b_i]$, $i \geq 1$,  where $b_1 = d+1-t$,  and for all $i$ where $b_i > 1$, 
$ b_{i+1} = \min  (\lceil b_i - cb_i ^{1/2} \rceil, b_i - 1)$.
  In this way, no two values $n_j', n_{j'}'$ can be in the same interval $I_i$ by the fact that $n_{j+1}' \leq n_{j}'-c(n_{j}')^{1/2}$ and the monotonicity of the function $\min(\lceil x - c\sqrt{x} \rceil, x-1)$ for integers $x \geq 1$. Assume $[1, d+1-t]$ is partitioned into $l$ intervals, then let $I_{l+1}=[-t+3,0]$. No two values $n_j', n_{j'}'$ can be in the interval $I_{l+1}$ because $n'_i$ is positive for every $i\neq r$.
  Thus the number of choices for $n'_{i_1},n'_{i_2},\cdots,n'_{i_{|L'|}}$ is at most 
  \begin{equation}
  \sum_{j_1 < \dots < j_{|L'|}} |I_{j_1}| \dots |I_{j_{|L'|}}| \label{eqn:Iiprod}
  \end{equation} 
 This is because we first need to pick the $|L'|$ different intervals $I_j$'s`. And once knowing $n_i'$ is in some interval $I_{j_i}$, there are at most $|I_{j_i}|$ ways to choose $n_i'$.

 Note that union of the disjoint intervals $I_i$, which is $[-t+3, d+1-t]$, has length $d-2$, and $[-t+3, d+1-t]$ is partitioned into $l+1$ intervals. By convexity, the quantity (\ref{eqn:Iiprod}) is upper bounded by $\binom{l+1}{|L'|} \left( \frac{d}{l'+1}  \right)^{|L'|}$. Furthermore, from the proof of Claim \ref{claim:boundr0}, we know $l+1 < r_0 = 4t^{1/2} \log_2^{1/4} t$. We have (\ref{eqn:Iiprod}), which is the number of possible subsequences of $n'_{i_1},n'_{i_2},\cdots,n'_{i_{|L'|}}$ where $i_j\in L'$,  is at most 
 \begin{equation}
     \binom{r_0}{|L'|} \left( \frac{d}{r_0}  \right)^{|L'|}.
  \label{eqn:Iiprod3}
 \end{equation}

Next, we select $n'_i$ for $i\in S'= \{i\in [r]| i-1 \in S \}$ one at a time. Assume $n'_{i-1}$ is already selected. By Corollary \ref{cor: n' fix v}, the vertex $v_{i-1}$ is already fixed and $d_{i-1}$ is known. For every $i\in S'$, we have $i-1\in S$ by the definition of $S$, we have $|Y_{i-1}|< M$. Because $|Y_{i-1}|=|L_{i-1}|-|D_{i-1}|=n'_{i-1}-n'_i-d_{i-1}$, we have $n'_i=n'_{i-1}-d_{i-1}-|Y_{i-1}|$.
Thus, $n'_i$ can be selected in the range $[n'_{i-1}-d_{i-1}-(M-1),n'_{i-1}-d_{i-1}].$ Thus, there are at most $M$ choices for each $n'_i$ for $i \in S'$.

Combining the results above, and notice that $v_1$ has $n$ choices, we have shown that, for a fixed $L\cup S$, 
$
|C(L)|\le n\cdot M^{|S'|}{r_0 \choose |L'|}\left(\frac{d}{r_0}\right)^{|L'|}.
$

For fixed $r,r_l$, we have  $|L'|=|L|=r_l$ and $|S'|=|S|-1=r-r_l-1$ by the discussion earlier. There are at most ${r-1 \choose r_l}$ ways to determine the partition $L\cup S$ of $[r]$ with $r \in S$ and $|L|=r_l$. Thus, the number of possible encoding of cliques $K$ with indices $r,r_l$ is at most 
$ n \cdot {r-1 \choose r_l}\cdot M^{r-r_l-1}{r_0 \choose r_l}(\frac{d}{r_0})^{r_l},$ as desired.
\end{proof}

 \begin{proof}[Proof of Lemma \ref{lem:very large counting}]
    We first group the cliques of order $k$ by the encoding $v_1, \dots, v_{r(K)}$, and then by the values $r(K)$ and $R_M(K)$.
    Similar to \cite{FW}, we will bound the number of $k$-cliques $K$ in $G$ with $r(k)=r<k$ and $R_M(K)=r_l$ by the product of the number of possible encoding of $k$-cliques with indices $r$ and $r_l$ proved in Lemma \ref{lem:rs code} and the number of $k$-cliques $K$ given $v_1, \dots, v_{r(K)}$ with $r(K) = r$ and $R_M(K) = r_l$. 
    
    Clearly $r(K)\le k$. If $r(K)= k$, the $k$-clique $K$ is completely determined by the encoding.
    Next, we will bound the number of $k$-cliques given a fixed encoding given $v_1, \dots, v_{r(K)}$ with $r(K) = r < k$ and $R_M(K) = r_l$. Recall that the encoding uniquely determines the terminal graph $G_{r(K)}$. We thus bound the number of cliques of order $k-r$ in $G_r$. Recall that $|V(G_r)| = n_r$.

    %For simplicity, we write $r = r(K)$ and $r_;=R_M(K)$ if it causes no confusion. 

For the cliques with $r(K)=r$ and $R_M(K)=r_l$, 
we split the cliques into three types: (i) those with $n_r \leq t-r$, (ii) those with $r=k$ and $n_r \geq t - r$, and (iii)  those with $r < k$, $n_r \geq t - r$, and $d_r \leq \frac{1}{2}(n_r  + r-t)^{1/2}$. By Claim \ref{claim: branch vtx set}, the maximum size of clique minor in $G_r$ is a most $t-r-s(K)\le t-r-s_M(r,r_l)$.

For type (i) where $n_r\leq t-r$, it is not hard to see $G_r\in \mathcal{H}^{t-r-s_M(r,r_l)}_{t-r}$ where recall $\mathcal{H}^s_m$ is the family of graphs $H$ with at most $m$ vertices and its clique minor has size at most $s$.
Thus, the number of cliques of order $k-r$ in the terminal graph $G_r$ which are of type (i) is at most $\mathcal{N}_{k-r}(\mathcal{H}^{t-r-s_M(r,r_l)}_{t-r})$.
For the cliques of type (ii), i.e., cliques of order $k$ with $r=k$ and with the same encoding, the encoding of length $r = k$ uniquely determines the clique $K$ of order $k$. 

Finally, we bound the number of cliques of type (iii), i.e., cliques with  $n_r > t-r$, $r<k$, and $d_r \leq \frac{1}{2}(n_r  + r-t)^{1/2}$. 

\begin{claim} \label{Gr is dense}
If we stop at step $r$ with graph $G_r$ such that $r< k$ and $n_r \geq t-r$, then $G_r\in \mathcal{G}_{t - r-s_M(r,r_l)+1}\cap \mathcal{D}$. 
\end{claim}
\begin{proof}

In this case, recall in the graph $G_r$ the vertex $v_r$ has the minimum degree. 
%and we remove $v_r$ and its non-neighbors $D_r$ and possibly some other vertices from $G_r$ to get $G_{r+1}$, which has $n_{r+1}$ vertices. 
Thus, the complement of $G_r$ has maximum degree $\Delta=|D_r|=d_r \leq \frac{1}{2}(n_r  + r-t)^{1/2} $ by the definition of $r(K)$. This means  in $G_r$
\begin{equation}
\Delta \leq  \frac{1}{2}(n_r  + r-t)^{1/2}. \label{eqn:deltasmall}
\end{equation}

Let $\omega$ be the clique number of $G_r$. Since $G$ has no $K_t$-minor and $G_r$ is in the common neighborhood of $v_1, \dots, v_r$, we need $\omega < t -r$, so 
$ n_r-\omega > n_r - (t-r) \geq  (2\Delta)^2 \geq 2\Delta^2+2 $
where the second inequality holds by (\ref{eqn:deltasmall}) and the last inequality holds when $\Delta \geq 1$. Hence, we have $G_r\in \mathcal{D}$ and the condition of Lemma \ref{lem:densehnumber} is satisfied. 
By Lemma \ref{lem:densehnumber}, the largest clique minor order in $G_r$ is $\lfloor \frac{n_r+\omega}{2} \rfloor\le t - r-s_M(r,r_l)$, and so $G_r\in \mathcal{G}_{t - r-s_M(r,r_l)+1}$ where $\mathcal{G}_s$ is the family of graphs $G$ such that $\lfloor \frac{|V(G)|+\omega(G)}{2} \rfloor \leq s-1$ where $\omega(G)$ is the order of the largest clique in $G$. 
\end{proof}

Therefore, the number of $k$-cliques $K$ of type (iii) (with indices $r$ and $r_l$) after fixing  $v_1, \dots, v_r$  is bounded above by
$\mathcal{N}_{k-r}(\mathcal{G}_{t - r-s_M(r,r_l)+1}\cap \mathcal{D})$. 
The next claim will give a range of $(r,r_l)$ where it is possible to have a $k$-clique $K$ with $r(K)=r$ and $R_M(K)=r_l$.

\begin{claim}\label{claim:slambda}
        Let $\lambda = t - k$. Let $K$ be a clique of order $k$ with pair of indexes $(r, r_l)$ and let $G_r$ be its terminal graph, it must holds that $s_M(r,r_l) \leq \lambda$.
    \end{claim}
    \begin{proof}
       We prove by contradiction. Suppose  $s_M(r,r_l)> \lambda$. From the Claim \ref{claim: branch vtx set}, the maximum size of clique minor in $G_r$ is a most $t-r-s(K)\le t-r-s_M(r,r_l)< k-r$, which implies we cannot find any clique of order $k-r$ in $G_r$. However, the subgraph induced by $K-\{v_1,v_2,\dots,v_{r-1}\}$ in $G_r$ is a clique of order $k-r$, a contradiction. 
    \end{proof}

Thus, combining the results on the cliques of types (i), (ii), (iii), the number of $k$-cliques with indices $r$ and $r_l$ after fixing $v_1, \dots, v_r$   is bounded above by $\mathcal{N}_{k-r}((\mathcal{G}_{t - r-s_M(r,r_l)+1}\cap \mathcal{D})\cup \mathcal{H}^{t-r-s_M(r,r_l)}_{t-r})$. Combining with the bound on the number of possible encoding $v_1, \dots, v_r$ by Lemma \ref{lem:rs code}, the desired quantity is proved, and here 
the summation over $r_l$ where $s_M(r, r_l) \leq \lambda$ is by Claim  \ref{claim:slambda}.
\end{proof}
 
\begin{rem}
    The reason we need to bound $\mathcal{N}_{k-r}( \mathcal{H}^{t-r-s_M(r,r_l)}_{t-r} )$ separately is because the optimizer candidate $T(t-r,t-r-s)$ is not in the family $\mathcal{G}_{t-r-s_M(r,r_l)+1}$ unless $s=0$ which will be discussed in the proof of the next Corollary \ref{lem:reduce2}.
\end{rem} 

Sometimes we do not need this elaborated upper bound in the Lemma \ref{lem:very large counting}. If we group cliques only by $r(K)$, the length of their encoding, then we can get the following cruder upper bound.

\begin{cor}\label{lem:reduce2}
Let $r_0 = 4 t^{1/2}{\log_2}^{1/4} t$. When $t$ is sufficiently large,
the maximum number of cliques of order $k$ in a $K_t$-minor free graph on $n$ vertices is at most
\[ 
n \cdot \sum^{\min(r_0, k))}_{r=1} \left(
 \binom{r_0}{r-1} \left( \frac{\beta t \sqrt{\ln t}}{r_0}  \right)^{r-1} \mathcal{N}_{k-r} (\mathcal{G}_{t-r+1}\cap \mathcal{D} )
\right).
\]
\end{cor}
\begin{proof}
 We could assume $t\ge r_0+1$. %Then, for any fixed $r\le r_0$, the graph $K_{t-r}$ exists.
 For any fixed $r\le r_0$, by the definition of $\mathcal{G}_t$ and $\mathcal{H}_m^s$, we have $\mathcal{G}_{t-r-s_M(r,r_l)+1} \subseteq \mathcal{G}_{t-r+1}$ and $\mathcal{H}^{t-r-s_M(r,r_l)}_{t-r} \subseteq \mathcal{H}^{t-r}_{t-r}$. Because $K_{t-r}\in \mathcal{H}^{t-r}_{t-r}$ and every $H\in \mathcal{H}^{t-r}_{t-r}$ is a subgraph of $K_{t-r}$, we have $\mathcal{N}_{k-r}(\mathcal{H}^{t-r}_{t-r})= \mathcal{N}_{k-r}(\{K_{t-r}\})$.
  It is not hard to check that $K_{t-r}$ is in $\mathcal{G}_{t-r+1}\cap \mathcal{D}$ by the definitions of $\mathcal{G}_t$ and dense graph. Thus, we have $\mathcal{N}_{k-r}( \mathcal{H}^{t-r-s_M(r,r_l)}_{t-r} ) \le \mathcal{N}_{k-r}(\mathcal{H}^{t-r}_{t-r}) \le \mathcal{N}_{k-r} (\mathcal{G}_{t-r+1}\cap \mathcal{D} )$ for any fixed $r\le r_0$. 

Now we bound the number of all possible encoding of cliques with index $r$ for any fixed $r\le r_0$. 
Set $M=0$ in the expression in  Lemma \ref{lem:very large counting} and then, among all $r_l\le r-1$, the only possibly non-zero summand is when $r_l=r-1$.  Thus the second sum only has one term where $r_l = r-1$. Plugging in $r_l = r-1$, the second sum equals to 
$
n {r_0 \choose r-1}\left( \frac{\beta t\sqrt{\ln t}}{r_0}\right) ^{r-1}  
\mathcal{N}_{k-r}((\mathcal{G}_{t-r-s_M(r,r-1)+1}\cap \mathcal{D}) \cup \mathcal{H}^{t-r-s_M(r,r-1)}_{t-r}) 
$. By the argument above, we have this quantity is at most 
$n\cdot \binom{r_0}{r-1} \left( \frac{d}{r_0}  \right)^{r-1} \mathcal{N}_{k-r} (\mathcal{G}_{t-r+1}\cap \mathcal{D} ) $. 
We can finish the proof by adding up the quantities among all possible values of $r$.
\end{proof}

\section{Asyptotic number of $k$-cliques in $K_t$-minor-free graphs }\label{sec:kmiddle}

In this section, we will apply Lemma \ref{lem:very large counting} and Corollary \ref{lem:reduce2} to prove Theorem \ref {thm: summary} for all $k$ such that $t-k\ll \log_2 t$.

\subsection{Asympototic number of $k$-cliques for large $k$}\label{sec:klarge}
In this subsection, we prove will Theorem \ref{thm:klarge0} which shows that when $k\geq 2t/3 + \tilde O(t^{1/2})$, the asymptotically maximum number of cliques in a graph on $n$ vertices with no $K_t$ minor is given by a graph which is a disjoint union of $T(t, t-1)$.

\begin{thm}  \label{thm:klarge0}

    Suppose $k \geq 2t/3+2\sqrt{t}{\log_2}^{1/4} t$. When $t$ is sufficiently large, the number of cliques of order $k$ in graphs on $n$ vertices and with no $K_t$-minor is at most

\[
n\cdot \left( \frac{\binom{t-1}{k}+\binom{t-2}{k-1}}{t} \right)\cdot t^{10\log_2 t}  \cdot 2^{\min \{4r_0 \log_2 t, 160 (t-k) \ln \ln t\}}
\]

\end{thm}

It is a simple computation to show that for any $k$ in the range above, $T^*_t(k)$ is the graph $K_t^-$, the complete graph $K_t$ minus an edge. We will show this result in the following lemma. 

\begin{lem}\label{lem:largekdense}
Fix $t$ and $k > 2(t+1)/3$.  % \lceil 2(t+1)/3 \rceil$.
The maximum possible number of cliques of order $k$ in $G$ among all $G \in \mathcal{G}_{t+1}$ is at most  $\binom{t}{k} + \binom{t-1}{k-1}$. This bound is sharp as the graph $K_{t+1}^-$ has $\binom{t}{k} + \binom{t-1}{k-1}$ cliques of order $k$. 
Moreover, we have $T^*_{t+1}(k)=K^-_{t+1}.$
\end{lem}

The number of $k$-cliques in $K_t^{-}$ is $ \binom{t-1}{k}+\binom{t-2}{k-1}$. Clearly $K_t^-$ is $K_t$-minor free. By considering $n/t$ disjoint copies of $K_t^{-}$, we thus have the following corollary which implies the Main Theorem \ref{thm: summary} when  $k \geq (t/3+r_0+2)/3$ and when $t-k \gg \log t$.

\begin{cor}[Corollary of Theorem \ref{thm:klarge0}]\label{cor: very large k}
Let $t$ be sufficiently large. Suppose $k \geq 2t/3+2\sqrt{t}{\log_2}^{1/4} t$  and $t - k \gg \log_2t $. Then the number of cliques of order $k$ in graphs on $n$ vertices and with no $K_t$-minor is at most
\[
     n\cdot \left( \frac{\binom{t-1}{k}+\binom{t-2}{k-1}}{t} \right)^{1+o_t(1)}=n\cdot \left(\frac{C^*_t(k)}{|V(T^*_t(k))|}\right)^{1+o_t(1)}.
    \] 
\end{cor}

\begin{proof}
     [Proof of Corolloary \ref{cor: very large k} form Theorem \ref{thm:klarge0} and Lemma \ref{lem:largekdense}]
Let $\lambda = t-k$.  Because  $k \geq 2t/3+2\sqrt{t}{\log_2}^{1/4} t>2t/3$, by Lemma \ref{lem:largekdense}, we have $T^*_t(k)=K_t^-$ and $C^*_t(k)=\binom{t-1}{k-1}+\binom{t-2}{k-2}$. In addition, $|T^*_t(k)|=t$.

By assumption, we have $\lambda \leq t/3$. Thus when $\lambda \gg \log_2 t$, we have $t^{10\log_2 t}  =  \binom{t}{\lambda}^{o_t(1)}$.   
     Similarly, when $\lambda\ge t^{1/2}{\log_2}^{3/2} t$, %we have $\lambda\gg 4t^{1/2}{\log_2}^{5/4}=r_0\log_2 t$. Thus, 
     we have $2^{4r_0 \log_2 t}\leq 3^{\lambda o_t(1)} \le \binom{t}{\lambda}^{o_t(1)}$ as $\lambda \leq t/3$. When $\lambda\le t^{1/2}{\log_2}^{3/2} t$, %we have $\lambda\le t^{2/3}$ and $\lambda \ln\ln t \ll \lambda(\log_2 (t^{1/3}))\le \lambda(\log_2 t-\log_2 \lambda)$. Thus, 
     we have $\lambda\le t^{2/3}$ and then $2^{160\lambda \ln\ln t} % = 2^{\lambda(\log_2 (t^{1/3}))o_t(1)} \le 2^{\lambda(\log_2 t-\log_2 \lambda)o_t(1)}
     = (\frac{t}{\lambda})^{\lambda o_t(1)} \le \binom{t}{\lambda}^{o_t(1)}$. 
     Thus, for every $\lambda \gg \log_2 t$, we have $2^{\min \{4r_0 \log_2 t, 160 \lambda \ln \ln t\}}  =  \binom{t}{\lambda}^{o_t(1)}$. 

    Because $t/3 \ge \lambda \gg \log_2 t$, we have $\binom{t-1}{\lambda-1}\ge \binom{t}{\lambda}/t> \binom{t}{\log_2 t}/t \ge \sqrt{t}^{\log_2 t}$. 
    Then we have $\binom{t}{\lambda}\le t\cdot\binom{t-1}{\lambda-1}= t^2\cdot \frac{\binom{t-1}{\lambda-1}}{t} \le  (\frac{\binom{t-1}{\lambda}}{t})^{o_t(1)}\cdot \frac{\binom{t-1}{\lambda-1}}{t} \le \left( \frac{\binom{t-1}{k}+\binom{t-2}{k-1}}{t} \right)^{1+o_t(1)}$. Combine this result with two results above, we have 
     \[
n\cdot \left( \frac{\binom{t-1}{k}+\binom{t-2}{k-1}}{t} \right)  \cdot t^{10\log_2 t}  \cdot 2^{\min \{4r_0 \log_2 t, 160 \lambda \ln \ln t\}}=n\cdot \left( \frac{\binom{t-1}{k}+\binom{t-2}{k-1}}{t} \right)^{1+o_t(1)}.
\] The corollary holds.
\end{proof}

Next, we will prove Lemma \ref{lem:largekdense}.

\begin{proof}[Proof of Lemma \ref{lem:largekdense}]
Let $f(t, k)$ be the quantity we want. We want to show $f(t, k) \leq \binom{t}{k} + \binom{t-1}{k-1}$ for $k > 2(t+1)/3$. 

For any graph $G$, for simplicity let $n(G)$ be the number of vertices in $G$, and $\omega(G)$ be the order of the largest clique in $G$. We also define
$x(G) = \lfloor( n(G) + \omega(G))/2\rfloor$.
Therefore a graph $G \in \mathcal{G}_{t+1}$ is equivalent to  $x(G) \leq t$. 
We prove the desired result by induction on $x(G) = t$. 

The base case is when $t \leq 3$.  When $t=1$, then $\lfloor (n+\omega)/2 \rfloor \leq 1$ and $k > 1$. We have $n = 2, \omega = 1$. This is a graph of two isolated vertices. The result also holds.

When $t = 2$, then $\lfloor (n+\omega)/2 \rfloor \leq 2$ and $k \geq 3$. Then $n + \omega = 4$ or $5$. Then the graph is either an edge or a path of two edges. The result clearly holds.

When $t = 3$, then $\lfloor (n+\omega)/2 \rfloor \leq 3$ and $k \geq 3$. Then $n + \omega = 6$ or $7$. Similarly, we can assume $\omega \geq 3$. So we have two options: $n = \omega = 3$, or $n = 4, \omega = 3$. The result clearly holds in the first case. 
For the latter case, the graph is a subgraph of $K_4^-$. In this case, the number of cliques of order 3 is at most 2. The result clearly holds.

\

Now we assume $t \geq 4$. Assuming the result holds for $x(G) = 1, 2, \dots, t-1$, we want to show it holds for $x(G) = t$. 

We first show if we are in the next two cases then we are done. 
In the graph $G$, let $d_v$ be the missing degree of $v$ in $G$ for any $v\in V(G)$.

\textbf{Case 1:}
Suppose there are two non-adjacent vertices $u, v\in V(G)$ with $d_u, d_v \geq 2$. 

There are three types of cliques of order $k$ in $G$:

Type 1: cliques not containing $u,v$. Then we count cliques of order $k$ in $G \setminus \{u,v\}$.
Since $x(G\setminus \{u,v\}) \leq x(G) - 1$, there are at most $f(t-1, k)$ of them.  Since $k > 2(t+1)/3$, then $k > 2(t-1 + 1)/3$.  Thus by inductive hypothesis,  $f(t-1, k) \leq \binom{t-1}{k} + \binom{t-2}{k-1}$.

Type 2: cliques containing $v$. Thus it does not contain the vertices not adjacent to $v$. Thus we count cliques of order $k-1$ in $G$ removing $v$ and the non-neighbors of $v$. Call this graph $G'$. Then $n(G') = n(G) - d_v - 1$, and $\omega(G') \leq \omega(G) - 1$.
Because $\lfloor (a-b)/2 \rfloor \le \lfloor a \rfloor -\lfloor b \rfloor$ for any integers $a\ge b$.
Therefore 
$ x(G') \leq x(G) - \lfloor (d_v + 2)/2 \rfloor$.
 Thus the number of cliques of type 2 is at most $f(t - \lfloor (d_v + 2)/2 \rfloor \rfloor, k-1)$.  In order to apply the inductive hypothesis, we need to check the condition
$k- 1 > 2( t - \lfloor (d_v + 2)/2 \rfloor +1 )/3$. It is true because $k\ge 2(t+1)/3$ and $2 \lfloor (d_v + 2)/2 \rfloor  / 3 \geq 1$ when $d_v \ge 2$.
Thus we can apply the inductive hypothesis for $f(t - \lfloor (d_v + 2)/2 \rfloor \rfloor, k-1)$.

Type 3 are the cliques containing $u$. Similarly, when $d_u \geq 2$, the number of cliques of type 3 is at most  $f(t - \lfloor (d_v + 2)/2 \rfloor, k-1)$ and this can be bounded by the inductive hypothesis.

Combining the three types  we just need to check 
\[  f(t-1,k) +  f(t - \lfloor (d_v + 2)/2 \rfloor, k-1) + f(t - \lfloor (d_u + 2)/2 \rfloor, k-1) \leq \binom{t}{k} + \binom{t-1}{k-1}.
\]
Since $d_u, d_v \geq 2$, by the inductive hypothesis, it suffices to show
\[   \binom{t-1}{k} + \binom{t-2}{k-1} +  2\left( \binom{t-2}{k-1} + \binom{t-3}{k-2}    \right) \leq \binom{t}{k} + \binom{t-1}{k-1}.
\]
By using $\binom{a}{b} - \binom{a-1}{b} = \binom{a-1}{b-1}$, it is equivalent to show
\[
2\left( \binom{t-2}{k-1} + \binom{t-3}{k-2}    \right) \leq  \binom{t-1}{k-1} + \binom{t-2}{k-2}.
\]
It suffices to  show the following two inequalities hold simultaneously: 
\[
2\binom{t-2}{k-1} \leq  \binom{t-1}{k-1}  \iff  t -1 \geq 2(t-k);
\]
and
\[
2\binom{t-3}{k-2} \leq  \binom{t-2}{k-2}  \iff  t -2 \geq 2(t-k);
\]
This holds when $t \geq 2$ since we have assumed $k > 2(t+1)/3$. 

\

\textbf{Case 2}

Now we suppose we have two adjacent vertices $u, v$ and there are  $d_u$ vertices not adjacent to $u$; and $d_v$ vertices not adjacent to $v$ in $G$. Let $d$ be the number of vertices not adjacent to either $u$ or $v$. Suppose $d_u, d_v \geq 1, d \geq 2$. 

Similar to before, we have four types of cliques to consider.

Type 1: cliques not containing $u,v$. Then we count cliques of order $k$ in $G \setminus \{u,v\}$. There are at most $f(t-1, k)$ of them. Since $k > 2(t+1)/3$, then $k > 2(t-1 +1)/3$. Thus  by the inductive hypothesis,  $f(t-1, k) \leq \binom{t-1}{k} + \binom{t-2}{k-1}$.

Type 2: cliques containing $u$ but not $v$. Thus it does not contain the vertices not adjacent to $u$. Thus we count cliques of order $k-1$ in $G$ removing the non-neighbors of $v$ and vertices $u,v$. Call this graph $G'$. Then $n(G') = n(G) - d_v - 2$, and $\omega(G') \leq \omega(G) - 1$. Thus the number of cliques of this type is at most $f(t - \lfloor (d_v + 3)/2 \rfloor, k-1)$. In order to apply the inductive hypothesis, we need to check 
$ k- 1 > 2( t - \lfloor (d_v + 3)/2 \rfloor + 1)/3$. It is true because $k\ge 2(t+1)/3$ and $2 \lfloor (d_v + 3)/2 \rfloor  / 3 \geq 1$ when $d_v \ge 1$.
Thus we can apply the inductive hypothesis for $f(t - \lfloor (d_v + 3)/2 \rfloor, k-1)$.

Type 3: Similarly, when $d_u \geq 1$, the number of cliques containing $u$ but not $v$ is at most $f(t - \lfloor (d_u + 3)/2 \rfloor, k-1)$ and to which we can apply the inductive hypothesis. 

Type 4: Cliques containing both $u,v$. We only need to count the number of cliques of order $k-2$ in $G'$ which is $G$ removing $\{u, v\}$ and the $d$ vertices not adjacent to either $u$ or $v$. $n(G') = n(G) - d - 2$, and $\omega(G') \leq \omega(G) - 2$.
The number of cliques of this type is bounded above by $f(t -\lfloor (4 + d)/2\rfloor, k-2)$. 
Again, it is not hard to check
$4k-2 > 2(t -\lfloor (4 + d)/2\rfloor + 1)/3$ when $d\ge 2$ and $k\ge 2(t+1)/3$.

We want to check 
\[f(t-1, k) + f(t - \lfloor (d_v + 3)/2 \rfloor, k-1) +f(t - \lfloor (d_u + 3)/2 \rfloor, k-1) + 
f(t -\lfloor (4 + d)/2\rfloor, k-2) \leq \binom{t}{k} + \binom{t-1}{k-1}. 
\]
Since we assumed $d_u, d_v \geq 1, d \geq \max(d_u, d_v)$ and $d \geq 2$, it suffices to prove 
\[f(t-1, k) +2 f(t -2, k-1)+ 
f(t -3, k-2) \leq \binom{t}{k} + \binom{t-1}{k-1}. 
\]
By plugging in the inductive hypothesis and the fact $\binom{a}{b} - \binom{a-1}{b} = \binom{a-1}{b-1}$, it suffices to prove
\[ 2\left( \binom{t-2}{k-1} + \binom{t-3}{k-2} \right) + \binom{t-3}{k-2} + \binom{t-4}{k-3}
\leq  \binom{t-1}{k-1} + \binom{t-2}{k-2}
\]
By subtracting $\left(\binom{t-2}{k-1} + \binom{t-3}{k-2}\right) $ from both sides and utilizing $\binom{a}{b} - \binom{a-1}{b} = \binom{a-1}{b-1}$,
it is equivalent to check
$ \left(\binom{t-2}{k-1} + \binom{t-3}{k-2}\right) +  \binom{t-3}{k-2} + \binom{t-4}{k-3}
\leq  \binom{t-2}{k-2} + \binom{t-3}{k-3}.
$
By a similar reasoning, it is equivalent to check
$ \binom{t-2}{k-1} + \binom{t-3}{k-2} 
\leq  \binom{t-3}{k-3} + \binom{t-4}{k-4}.$
It is easy to check both holds when $k > 2(t+1)/3$ and $t \geq 4$.

Therefore if in $G$, there are two non-adjacent vertices both having missing degrees at least 2 then we are done by Case 1. Hence if there are at least two vertices with missing degrees at least 2, they are adjacent. But then we are done by Case 2. As a consequence at most one vertex in $G$ has a missing degree of at least 2, call it $v$ if it exists.  

Notice that
 all the vertices adjacent to $v$ should have zero missing degrees since otherwise we are done again by Case 2. Thus if $v$ with missing degree at least 2 exists, the complement of $G$ is a star with center $v$ with some isolated vertices. Suppose this star has $z$ edges. 
 Then in $G$ with $n$ vertices, $\omega(G) = n-1$. By assumption $\lfloor (n + (n-1))/2 \rfloor \leq t$ which means $n \leq t+1$.  
 The number of cliques of order $k$ is $\binom{n-1}{k} + \binom{n- z-1}{k-1}$ where the first term is when the center vertex $v$ of the missing star is not picked for the clique; the second term is when the center vertex $v$ of the missing star is picked for the clique, and thus all the $z$ non-neighbors of the vertex cannot be part of the clique.
 This is maximized when $n = t+1$ and $z= 2$, and this is strictly smaller than the bound desired. 
 
 Therefore there is no vertex of missing degree at least 2. Thus the complement of $G$ is a matching. 
However, if the matching has at least two edges, then we are in Case 2 again. Therefore the matching has exactly one edge. Thus again $G$ has $n \leq t+1$ vertices. The number of cliques of order $k$ in $G$ is at most
$
\binom{n-1}{k} + \binom{n- 2}{k-1} \leq \binom{t}{k} + \binom{t-1}{k-1}.
$ Thus the desired upper bound is proved; and it can be achieved if and only if in this case where $n(G)=t+1$ and $G$ has only one missing edge, i.e., $G\cong K_{t+1}^- $. 

Moreover, by the definition of $\mathcal{G}_t$, every Tur\'an graph $T(2t-\omega+1,\omega)$ with $\omega \le t$ is a graph in $\mathcal{G}_{t+1}$. The optimal graph $K^-_{t+1}$ in $G_{t+1}$ also maxized the number of cliques of order $k$ among Tur\'an graphs $T(2t-\omega+1,\omega)$ for every $\omega \le t$. Thus, by definition of $T^*_t(k)$, we have $T^*_{t+1}(k)=K^-_{t+1}$.
\end{proof}

The rest of this subsection dedicates to the proof of Theorem \ref{thm:klarge0}. 
To prove this theorem, we will apply Lemma \ref{lem:very large counting}. 
To apply Lemma \ref{lem:very large counting},
we want to bound $\mathcal{N}_{k-r}(\mathcal{G}_{t-r-s(r,r_l)+1}\cup \mathcal{H}^{t-r-s(r,r_l)}_{t-r})$, which is the 
maximum possible number of cliques of order $k-r$ in a dense graph $G_r \in  \mathcal{G}_{t-r-s(r,r_l)+1}\cup \mathcal{H}^{t-r-s(r,r_l)}_{t-r}$.
Because $k \geq 2t/3+2\sqrt{t}{\log_2}^{1/4} t$, we will see that  $\mathcal{N}_{k-r}(\mathcal{G}_{t-r-s(r,r_l)+1})$ can be bounded tightly by Lemma \ref{lem:largekdense}. To prove Theorem \ref{thm:klarge0}, we also need to bound $\mathcal{N}_{k-r}(\mathcal{H}^{t-r-s}_{t-r} )$ by the following Lemma \ref{lem: count Ht}.

\begin{lem} \label{lem: count Ht}
When $t\ge 6r_0$ (or $t\ge 2000$) and $\lambda=t-k \le t/3$, for every $r\le r_0 $, we have    $\mathcal{N}_{k-r}(\mathcal{H}^{t-r-s}_{t-r} ) \le  \binom{t-r-s}{k-r}\cdot 2^s$.
\end{lem}
\begin{proof}
When $s> \lambda$, there is no clique of order $k-r$ in $\mathcal{H}^{t-r-s}_{t-r}$, so this statement is trivially true. Now assume $s\le \lambda =o(t)$.
Zykov's theorem \cite{zykov} states that the graph on $n$ vertices without $K_{\omega+1}$-subgraph and with the most number of $(k-r)$-cliques is achieved by the Tur\'an graph $T(n, \omega)$. Thus, we only need to bound the number of $k$-cliques in $T(t-r,t-r-s)$. Because $t-r-s\ge t-r_0-\lambda \ge t/2$, then each part of the Tur\'an graph $T(t-r,t-r-s)$ has size $1$ or $2$. Also, there are $s$ parts that have sizes of $2$, and $t-r-2s$ parts have sizes of $1$.

Any two vertices in a clique $K_{k-r}$ in $T(t-r,t-r-s)$ can not belong to the same part of $T(t-r,t-r-s)$ as each part of a Tur\'an graph is an independent set.
For any given $k-r$ distinct parts of $T(t-r,t-r-s)$, there are at most $2^s$ distinct copies of $K_{k-r}$ using these parts since each part has at most 2 vertices and there are $s$ parts of size 2 in the Tur\'an graph.
Thus, the number of cliques of order $k-r$ in $T(t-r,t-r-s)$ is at most $\binom{t-r-s}{k-r}\cdot 2^s$.
\end{proof}

\begin{claim} \label{compu: H}
    For any fixed $t$ and $k$ such that $\lambda=t-k \le t/3$, let $f_{t,k}(s)=\binom{t-s}{k}\cdot 2^s$. Then $f_{t,k}(s)$ is strictly decreasing for $s\in [0,\lambda]$. 
\end{claim}
\begin{proof}
    For any $s\in [0,\lambda-1]$, we have 
    $
    \dfrac{f_{t,k}(s+1)}{f_{t,k}(s)}=2\cdot \dfrac{t-(s+1)-k+1}{t-s}=2\cdot \dfrac{\lambda-s}{t-s}\le \dfrac{2\lambda}{t}<1.
    $ 
\end{proof}

Now we will proceed to prove Theorem \ref{thm:klarge0}:

\begin{proof}
[Proof of Theorem \ref{thm:klarge0}]
 When $t$ is sufficiently large, we will use the bound in Lemma \ref{lem:very large counting} to bound the maximum number of cliques of order $k$ in a graph without a $K_t$-minor. 
 It is easy to see that \begin{equation} 
    \binom{r_0}{r-1} \left( \frac{\beta t \sqrt{\ln t}}{r_0}  \right)^{r-1} \le \binom{\beta t \sqrt{\ln t}}{r-1}.
    \label{ineq:crude choose}
\end{equation}Applying (\ref{ineq:crude choose}) to Lemma \ref{lem:very large counting},  our goal is to bound the quantity
\[ 
n\cdot r_0^2 \cdot   \left( \max_{(r, r_l): s_M(r,r_l)\le \lambda}
 {r \choose r-r_l}M^{r-r_l}{\beta t\sqrt{\ln t} \choose r_l}  \mathcal{N}_{k-r}(\mathcal{G}_{t-r-s(r,r_l)+1}\cup \mathcal{H}^{t-r-s(r,r_l)}_{t-r}) \right).
\]

Let $s_M(r_l)= \lfloor r_l-1 - 7\cdot  \left(\log_{\frac{1}{1-\epsilon}}d+(8r_l\cdot \log_{\frac{1}{2\epsilon}} M)/{M}\right) \rfloor$. When $t$ is sufficiently large, by Lemma \ref{lem:nonempty Y}, we have $s_M(r,r_l)\ge s_M(r_l) $.
When $\lambda\gg \log t$, if $r\ge 4 \lambda$, we have $s_M(r,r_l)> \lambda$ by Lemma \ref{lem: half branch} which contradicts with Claim \ref{claim:slambda}. Thus, we only consider the term with $r_l< r\le \min \{r_0, 4 \lambda\}$, and in this range, we have
\[
r_l -s_M(r_l)\le 8 \left(\log_{\frac{1}{1-\epsilon}}d+(32\lambda  \log_{\frac{1}{2\epsilon}} M)/{M}\right). 
\] 

 Because $k > (2t + r_0 +2)/3$, we have $k > (2t + r +2)/3$ and then $(k-r) > 2(t-r +1)/3$ for all $r \leq r_0-1$. By Lemma \ref{lem:largekdense}, we have:    
\[ 
 \mathcal{N}_{k-r}(\mathcal{G}_{t-r-s(r,r_l)+1})
 \le {{t-r-s_M(r,r_l)} \choose k-r } +{{t-r-s_M(r,r_l)-1} \choose k-r-1} \le 2{{t-r-s_M(r_l)} \choose k-r }.
\]
Because $(k-r) > 2(t-r +1)/3$ and $r_0\le t/6$ when $t$ is sufficiently large, and $s_M(r,r_l)\ge s_M(r_l) $, by Lemma \ref{lem: count Ht} and Claim \ref{compu: H}, we have
\[ 
 \mathcal{N}_{k-r}(\mathcal{H}^{t-r-s(r,r_l)}_{t-r}))
 \le {{t-r-s_M(r,r_l)}  \choose k-r }\cdot 2^{s_M(r,r_l)} \le {{t-r-s_M(r_l) } \choose k-r }\cdot 2^{s_M(r_l)}. 
\] 

For our convenience, let $r_\lambda=\min \{r_0, 4 \lambda\}$. With the fact that $\binom{cn}{k}\le c^k \binom{n}{k}$ for any integers $n\ge k\ge 0$ and $c\ge 1$, we have the maximum number of cliques of order $k$ in a graph without $K_t$ as a minor is at most

\begin{align*}
     & n\cdot r_0^2 \cdot  \max_{r_l< r\le r_\lambda } \left(
 {r \choose r_l}M^{r-r_l}{t \choose r_l} \cdot (\beta\sqrt{\ln t})^{r_l} \cdot \binom{t-r-s_M(r_l)}{ k-r } \cdot 2^{(1+s_M(r_l))} \right)  \\
  \le & n\cdot r_0^2 \cdot  \max_{r_l< r\le r_\lambda} \left(
 {r_\lambda \choose r_l} \cdot M^{r_\lambda}{t \choose s_M(r_l)+(r_l-s_M(r_l))} \cdot (\beta\sqrt{\ln t})^{r_\lambda} \cdot \binom{t-r-s_M(r_l)}{ (t-r -s_M(r_l))-(k-r)} \cdot 2^{ r_l} \right)\\
 \le & n\cdot r_0^2 \cdot  M^{r_\lambda} \cdot (\ln t)^{r_\lambda} \cdot \max_{r_l< r\le r_\lambda} \left(
 2^{r_\lambda} \cdot {t \choose s_M(r_l)+(r_l-s_M(r_l))} \cdot \binom{t-r-s_M(r_l)}{ \lambda -s_M(r_l)} \cdot 2^{ r_\lambda}  \right) \tag{A}\\
 \le & n\cdot r_0^2 \cdot  (4M)^{r_\lambda} \cdot (\ln t)^{r_\lambda}  \max_{r_l< r_0: s_M(r_l)\le \lambda} \left(
 {t \choose s_M(r_l)}\cdot t^{(r_l-s_M(r_l))} \cdot \binom{t}{ \lambda -s_M(r_l)}  \right) \tag{B}\\
 \le & n\cdot r_0^2 \cdot  (4M)^{r_\lambda} \cdot (\ln t)^{r_\lambda} \cdot \max_{r_l< r_0: s_M(r_l)\le \lambda} \left(   {t \choose s_M(r_l)} \cdot \binom{t}{ \lambda - s_M(r_l) }   \right) \cdot t^{8 \left(\log_{\frac{1}{1-\epsilon}}d+(32\lambda  \log_{\frac{1}{2\epsilon}} M)/{M}\right)}\\
 \le &n\cdot r_0^2  \cdot (4M)^{r_\lambda} \cdot (\ln t)^{r_\lambda}\cdot 2^{\min \{ 2r_0\log_2 t, 4\lambda\}} \cdot {t \choose \lambda} \cdot t^{8 \left(\log_{\frac{1}{1-\epsilon}}d+(32\lambda  \log_{\frac{1}{2\epsilon}} M)/{M}\right)}. \tag{C}
\end{align*}

Inequality (A) holds because  $\beta<1$ and $\binom{n}{k}\le 2^n$ for every $n\ge k\ge 0$, and $(t-r-s_M(r_l))-(k-r)=t-k-s_M(r_l)=\lambda-s_M(r_l)$.
Inequality (B) holds because $\binom{t-r-s_M(r_l)}{ \lambda -s_M(r_l)}=0$ if $s_M(r_l)> \lambda$. Therefore the maximum happens at a value of $r_l$ where $s_M(r_l) \leq \lambda$. It also used the inequality
$ 
 \binom{n}{a+b} 
 \le \binom{n}{a}\cdot n^b 
$.

For inequality (C), we use the following inequality: for any $1 \leq a\le \lambda$,  
$
\binom{t}{\lambda-a} \binom{t}{a}\le (\frac{et}{\lambda-a})^{(\lambda-a)} (\frac{et}{a})^a\le (\frac{e\lambda}{\lambda-a})^{(\lambda-a)} (\frac{e\lambda}{a})^a (\frac{t}{\lambda})^\lambda \le e^{\lambda}e^{\lambda}\binom{t}{\lambda}.
$
Here the last inequality holds because $f(x)=(\frac{en}{x})^x$ is increasing in the range $x\in [1,n]$ and is decreasing when $x\ge n$, and we have $f(x)\le e^n$. Thus, we have ${t \choose s_M(r_l)} \cdot \binom{t}{ \lambda - s_M(r_l) }  \le e^{2\lambda} \binom{t}{\lambda}$. Because $s_M(r_l)\le r_l< r_0$ and because $\binom{n}{k}\le t \binom{n}{k+1}$, we can also show that ${t \choose s_M(r_l)} \cdot \binom{t}{ \lambda - s_M(r_l) }  \le t^{s_M(r_l)} \cdot t^{s_M(r_l)} \binom{t}{\lambda} \le t^{2r_0}\binom{t}{\lambda}$. Thus, we have ${t \choose s_M(r_l)} \cdot \binom{t}{ \lambda - s_M(r_l) }  \le 2^{\min \{ 2r_0\log_2 t, 4\lambda\}} \binom{t}{\lambda}$.

In the rest of this proof, we assume $t$ is sufficiently large.
As $\epsilon< \frac{1}{2}$, we have $r_0^2 \cdot t^{8 \log_{\frac{1}{1-\epsilon}}d} \le t^{9\log_2 t}$. Set $M=(\ln t)^2$, and then we have 
$
    t^{256 \left((\lambda  \log_{\frac{1}{2\epsilon}} M)/{M}\right)}=2^{o(\lambda)}.
$
After setting $M=(\ln t)^2$, we have  $(4M)^{r_\lambda}\cdot (\ln t)^{r_\lambda}\le (\ln t)^{10r_\lambda}=2^{20r_\lambda(\log_2 \ln t)}$.

When $\lambda\ge \frac{1}{2}r_0 \log_2 t$, we have $\lambda\ge r_0/4$ and $20r_\lambda(\log_2 \ln t) = 20r_0(\log_2 \ln t)<r_0 \log_2 t$.
Thus, the bound above is $n\cdot \binom{t}{\lambda}\cdot 2^{4r_0 \log_2 t}$. 
When $\lambda \le \frac{1}{2}r_0 \log_2 t$, we have $\min \{ 2r_0\log_2 t, 4\lambda\}=4\lambda\le \lambda \ln \ln t$.
We also have $20r_\lambda \log_2 \ln t\le 80\log_2 e \cdot\lambda \ln \ln t $.
Thus, the bound above is $n\cdot \binom{t}{\lambda}\cdot 2^{160 \lambda \ln \ln t}\cdot t^{9\log_2 t}$.
Because $\binom{t}{\lambda}=\binom{t}{k}\le t\cdot \binom{t-1}{k}$, the bound above is 
\[
n\cdot \left( \frac{\binom{t-1}{k}+\binom{t-2}{k-1}}{t} \right) \cdot t^2 \cdot t^{9\log_2 t} \cdot 2^{\min \{4r_0 \log_2 t, 160 \lambda \ln \ln t\}}.
\]
Then the Theorem \ref{thm:klarge0} follows as $t^2< t^{\log_2 t}$ when $t$ is sufficiently large.
\end{proof}

\subsection{Asympototic number of $k$-cliques for $k$ in the middle range}

In this subsection, we will give a bound in the following theorem which can prove our main Theorem \ref{thm: summary} for $k$ such that $\min{(k,t-k)} \gg O(t^{1/2}{\log_2}^{5/4} t)$.
For fixed $t$ and $k$, recall that {$T^*_t(k)$ is the Tur\'an graph  $T(2t-\omega-1,\omega)$ maximizing the number of cliques of order $k$  
among all $\omega$ such that $k\le \omega \le t-1$, and $C^*_t(k)$ denoted the number of cliques of order $k$ in $T^*_t(k)$.

\begin{thm}\label{thm:main}

When $t$ is sufficiently large and $\min{(k,t-k)} \gg O(t^{1/2}{\log_2}^{5/4} t)$, the number of cliques of order $k$ in a $K_t$-minor free graph on $n$ vertices is at most 
\[
n  \cdot \frac{C_t^*(k)}{|T^*_t(k)|} \cdot 2^{8t^{1/2}{\log_2}^{5/4} t }.
\]

\end{thm}
\begin{rem}
We will check in the next corollary that this bound is asymptotically sharp up to multiplicative error $2^{O(t^{1/2}{\log_2}^{5/4} t) }$ for every $k$ in this range when considering disjoint copies of the graph $T_t^*(k)$ which will be proved to be $K_t$-minor free in the next the Proposition \ref{prop:strucmiddle}.
\end{rem}
\begin{rem}
     This bound is true for $k$ in any range. However, when $k$ is too large or too small, the error term in this bound will be much larger than the main term.
\end{rem}

To prove Theorem \ref{thm:main} from Corollary \ref{lem:reduce2} in this range of $k$, we need 
to get a better understanding of the bound $\mathcal{N}_{k}(\mathcal{G}_{t}\cap \mathcal{D} )$, we will prove the following proposition which shows the graph in $\mathcal{G}_{t}$ that achieved the maximum number of $k$-cliques is a $K_t$-minor free Tur\'an graph $T^*_t(k)$. This directly implies Proposition \ref{prop:Tstarstructure}.

\begin{prop}\label{prop:strucmiddle}
 Among all the graphs $G\in \mathcal{G}_t$, the one that maximizes the number of cliques of order $k$  is the Tur\'an graph $T^*_t(k)$. Thus,  $\mathcal{N}_k(\mathcal{G}_t \cap \mathcal{D} ) \le C^*_t(k)$. In addition, $T^*_t(k)$ is $K_t$-minor free. Quantitatively, we have the following bound of $C^*_t(k)$:

 \[\binom{t-1}{k} \max\left(1,  \left( 2- 4\sqrt{2k/t}\right)\right)^k \leq  C^*_t(k) \leq \binom{t-1}{k} 2^k. \]
 Moreover, the number of parts $\omega$ in $T^*_t(k)$ is bounded by $\sqrt{tk}/4\le \omega \le 10 \sqrt{tk}$.
\end{prop}

 The following corollary will imply the Main Theorem \ref{thm: summary} when $\min(k,t-k) \gg O(t^{1/2}{\log_2}^{5/4} t) $.

\begin{cor}[Corollary of Theorem \ref{thm:main}]\label{cor:main} 

Suppose $\min{(k,t-k)} \gg O(t^{1/2}{\log_2}^{5/4} t)$. The number of cliques of order $k$ in a $K_t$-minor free graph on $n$ vertices is at most 
$
 n  \cdot \left( \frac{C_t^*(k)}{|T^*_t(k)|}\right)^{(1+o_t(1))}.
$

\end{cor}
\begin{proof}[Proof of Corollary assuming Theorem \ref{thm:main} and Proposition \ref{prop:strucmiddle}]
By the definition of $T_t^*(k)$, we have $|T^*_t(k)|\le 2t$. Furthermore, $C^*_t(k) \ge \binom{t-1}{k}$ by Proposition \ref{prop:strucmiddle}. Thus, ${C^*_t(k)}/{|T^*_t(k)|}\ge {\binom{t-1}{k}}/{2t} $.
   Because $\min(t,t-k) \gg \log_2 t$, we have $\binom{t-1}{k}\ge \binom{t}{k}/t >\binom{t}{\log_2 t}/t \ge \sqrt{t}^{\log_2 t}$. Then we have 
    $\binom{t}{k}\le 2t^2\cdot (\frac{\binom{t-1}{k}}{2t})  \le  (\frac{\binom{t-1}{k}}{2t})^{1+o_t(1)}$. 
    Let $m=\min (k,t-k)$. As $m\le t/2$, we have $2^{8t^{1/2}{\log_2}^{5/4} t }=2^{mo_t(1)}\le 2^{m(\log_2 t -\log_2 m)o_t(1)}=\binom{t}{m}^{o_t(1)}=\binom{t}{\lambda}^{o_t(1)}$. Thus, we have $2^{8t^{1/2}{\log_2}^{5/4} t }=\left(\frac{\binom{t-1}{k}}{2t} \right)^{o_t(1)}\le \left( \frac{C_t^*(k)}{|T^*_t(k)|}\right)^{o_t(1)}$. We can finish the proof by applying this inequality to the Theorem \ref{thm:main}.
\end{proof}

\subsubsection{Proof of Proposition \ref{prop:strucmiddle}}

Recall that ``dense" in our application means the condition as in Lemma \ref{lem:densehnumber} is satisfied, i.e., when the maximum missing degree $\BD$ in a graph $G$ is such that
\begin{equation} |V(G)| \geq  \omega(G) + 2 \BD^2 + 2 \text{ or } \BD \leq 1. \label{eqn:dense}
\end{equation}

Recall that $\mathcal{G}_t$ is the family of graphs $G$ such that $\lfloor \frac{|V(G)|+\omega(G)}{2} \rfloor \leq t-1$. Note that there can be graphs that are not dense but also belong to $\mathcal{G}_t$.
{We will call $G$ an \textit{optimizer in $\mathcal{G}_t$} if it maximizes the number of $k$-cliques among all graphs in $\mathcal{G}_t$.}
Let a \emph{balanced complete multipartite graph} be a complete multipartite graph where the orders of each part differ by at most 1. Clearly all Tur\'an graphs are balanced complete multipartite graphs.

The proof goes as follows. In the whole subsection, we always assume $t$ is large. 
\begin{enumerate}
\item We first show that the optimizer $G$ in $\mathcal{G}_t$ is given by Tur\'an graphs, which are complete multipartite graphs where the orders of different parts differ by at most 1.  (Lemma \ref{lem:bcm}). Furthermore, we show $|V(G)| + \omega(G) = 2t-1$ and $G=T^*_t(k)$. (Claim \ref{claim:full}). 
\item Next, we will show that every balanced complete multipartite graph satisfying $|V(G)| + \omega(G) = 2t-1$ 
is $K_t$-minor free. (Lemma \ref{lem: kt minor free}).

\item To illustrate the structure of the optimizer Tur\'an graph $G$, we prove an upper bound of the maximum missing degree of $G$. (Lemma \ref{lem:smalla}).

\item If $T_t^*(k)$ is the optimizer, we obtain asymptotically the value for $C_t^*(k)$ by proving a simple upper bound (Claim  \ref{claim:simpleckbound}) and constructing a  lower bound (Lemma \ref{lem:ckapprox}).
\item By further comparing the number of $k$-cliques in $T(n,w)$ to the upper and lower bounds for $C_t^*(k)$ as mentioned above, we are able to determine asymptotically the number of parts in $T_t^*(k)$ (Lemma \ref{lem:tstar}). 
\end{enumerate}

Fix $n, \omega$ such that $\lfloor (n + \omega)/2 \rfloor \leq t-1$. Zykov's theorem \cite{zykov} states that the graph on $n$ vertices without $K_{\omega+1}$-subgraph and with the most number of $k$-cliques is achieved by the Tur\'an graph $T(n, \omega)$, which is a balanced complete multipartite graph on $n$ vertices and with $\omega$ parts. Therefore we have the following simple lemma.

\begin{lem}\label{lem:bcm}
For all graphs $G \in \mathcal{G}_t$, the ones with the maximum number of cliques of order $k$ are balanced complete multipartite graphs.
\end{lem}

By Lemma \ref{lem:bcm}, to determine the optimizer in $\mathcal{G}_t$, we only need to consider balanced complete multipartite graphs. Suppose the optimizer $G$ is a graph with $l$ parts and the parts $A_1,\dots,A_l$ have orders $a_1, \dots, a_l$ respectively. Thus $\omega(G)=l$. Also, because $K_t^-\in \mathcal{G}_t$ contains some $k$-clique, we need $l=\omega(G)\ge k$. Otherwise, $G$ does not contain any $k$-cliques which contradicts the definition of the optimizer. 

\begin{claim}\label{claim:full}
The graph $G \in \mathcal{G}_t$ which maximizes the number of $k$-cliques satisfies
$ \sum_{i=1}^l a_i + l = 2t-1$.
\end{claim}
\begin{proof}

For $G\in \mathcal{G}_t$, we need $\lfloor (n+\omega)/2 \rfloor \leq t-1$. In this case, $n = \sum a_i$ and $\omega(G)=l$. Thus the condition is equivalent to $\lfloor ( \sum_{i=1}^l a_i + l)/2 \rfloor  \leq t-1$.
Therefore $ \sum_{i=1}^l a_i + l \leq 2(t-1)+1$. 
 Suppose the claim does not hold, i.e., $ \sum_{i=1}^l a_i + l < 2(t-1)+1$. We will produce a new graph $G'$ in $\mathcal{G}_t$ but with more $k$-cliques. Let $G'$ be given by $a_1' = a_1 + 1, a_i' = a_i$ for $1 < i \leq l$. In this case, 
  $(\sum_{i=1}^l a_i' + l )/2< (2(t-1)+1 + 1)/2 = t.$
Since $t$ is an integer, 
$\lfloor (\sum_{i=1}^l a_i' + l )/2 \rfloor \leq t-1.$
This means $G' \in \mathcal{G}_t$. While on the other hand, {because $l\ge k$}, the number of cliques of order $k$ in $G'$ is strictly larger than the value for $G$, a contradiction.
\end{proof}

This claim shows that the optimizer $G$ in $\mathcal{G}_t$ is a Tur\'an graph $T(2t-l-1,l)$ for some $l\le t-1$. Thus, we have $G=T^*_t(K)$. Next, we show $G$ is $K_t$-minor free by the following lemma.  

\begin{lem} \label{lem: kt minor free}
For any $t\ge 2$ and $l\le t-1$, the Tur\'an graph $T(2t-l-1,l)$ does not contains a $K_t$-minor.    
\end{lem}
\begin{proof}
   We will prove this statement by contradiction. Suppose the statement is not true, then there exists $t\ge 2$ and $l\le t-1$ such that $G=T(2t-1-l,l)$ contains a $K_t$-minor. Let the vertex set of this $K_t$-minor be $\{v_1,v_2,\cdots,v_s\}\cup B_1\cup B_2 \cup \cdots \cup B_{t-s} $ where $v_i\in V(G)$ and $|B_i|\ge 2$ and $B_i$ is contracted to be a vertex in this $K_t$-minor. For every $i,j\in [s]$ with $i\neq j$, we have $v_iv_j\in E(G)$. Thus, $v_i$ and $v_j$ can not belong to the same part of $G=T(2t-l-1,l)$ as each part of a Tur\'an graph is an independent set. Thus, we have $s\le l$ and the $K_t$-minor has at least $s+2(t-s)=2t-s\ge 2t-l$ vertices which contradicts with $|V(G)|\le 2t-l-1$.    
\end{proof}

To find the optimizer in $\mathcal{G}_t$, we are doing the following integer optimization problem to find solutions $\{a_i\}_{i\le l}, l$: 
\begin{align}
\max  \sum_{1\leq i_1 < \dots < i_k \leq l} a_{i_1} \dots a_{i_k}. \label{eqn:ckG''2} \\
\text{ s. t. } a_i \geq 1, l \geq 1, \sum_{i=1}^l a_i + l = 2t-1. \label{eqn:constraint1}
\end{align}
The objective function (\ref{eqn:ckG''2}) is the number of $k$-cliques in the (balanced) complete multipartite graph with part $i$ has order $a_i$ and in total $l$ parts. The constraint (\ref{eqn:constraint1}) is from Claim \ref{claim:full}.

Next, we prove the following lemma which gives an upper bound of the maximum missing degree $\max a_i-1$ of $G$.
\begin{lem}\label{lem:smalla}
In the optimal complete multipartite $G$ satisfying the constraint (\ref{eqn:constraint1}) and optimizing (\ref{eqn:ckG''2}), the order of each part $a_i$ satisfies $(a_i-1)^2 < \frac{4n-3k+7-4a_i}{k-1}
$ or $a_i < 3$, where $n$ is the number of vertices in $G$. 
\end{lem}

\begin{rem}
When $k\ge 25$, by this lemma, we can show $G$ is a dense graph. This means $\mathcal{N}_{k}(\mathcal{G}_t \cap \mathcal{D})=C^*_t(k)=\mathcal{N}_{k}(\mathcal{G}_t )$ when $k\ge 25$.
\end{rem}

\begin{proof}

Suppose $a_1 \geq 3$. Let $G'$ be a graph with parts $B_1, B_2, A_2, \dots, A_l$ where $A_2, \dots, A_l$ are the same as in $G$, and the part $A_1$  in $G$ splits into $B_1, B_2$, where $|B_1| + |B_2| =|A_1| -1 = a_1 - 1$, and $|B_1|, |B_2| \geq 1$. This is possible since $a_1 \geq 3$. Clearly $G'$ also satisfies (\ref{eqn:constraint1}). 
Let $|B_1| = b_1, |B_2| = b_2$. 

Then the objective function for $G'$, i.e., number of cliques of order $k$ in $G'$ can be written as 
\[ \sum_{2 \leq i_1 < \dots < i_k \leq l}  a_{i_1} \dots a_{i_k}
+ 
\sum_{2 \leq i_2 < \dots < i_k \leq l} (b_1 + b_2) a_{i_2} \dots a_{i_k} + 
\sum_{2 \leq i_3< \dots < i_k \leq l} (b_1 b_2) a_{i_3} \dots a_{i_k}\]
where the first term counts the number of cliques of order $k$ not containing a vertex in $B_1 \cup B_2$, while the second term counts the ones containing exactly one vertex from $B_1 \cup B_2$, and the third term counts the ones containing one vertex in $B_1$ and one vertex in $B_2$.

Recall the number of cliques of order $k$ in $G$ is exactly 
\[ \sum_{2 \leq i_1 < \dots < i_k \leq l}  a_{i_1} \dots a_{i_k}
+ 
\sum_{2 \leq i_2 < \dots < i_k \leq l} a_1 a_{i_2} \dots a_{i_k} 
\] where the first term is the number of cliques of order $k$ in $G$ with no vertex in $A_1$, and the second term is the ones with a vertex in $A_1$.

As $G$ has more $K_k$ than $G'$, the number of $K_k$ in  $G$ minus the one in $G'$ satisfies
\begin{align}
& \sum_{2 \leq i_2 < \dots < i_k \leq l} a_1 a_{i_2} \dots a_{i_k}
-\sum_{2 \leq i_2 < \dots < i_k \leq l} (b_1 + b_2) a_{i_2} \dots a_{i_k} - 
\sum_{2 \leq i_3< \dots < i_{k} \leq l} (b_1 b_2) a_{i_3} \dots a_{i_k}  \nonumber \\
=&  \sum_{2 \leq i_2 < \dots < i_k \leq l} (a_1 - b_1 - b_2) a_{i_2} \dots a_{i_k} - \sum_{2 \leq i_3< \dots < i_{k} \leq l} (b_1 b_2) a_{i_3} \dots a_{i_k} \nonumber \\
=& \sum_{2 \leq i_2 < \dots < i_k \leq l}  a_{i_2} \dots a_{i_k} -(b_1 b_2) \sum_{2 \leq i_3< \dots < i_{k} \leq l}  a_{i_3} \dots a_{i_k} \geq 0, \label{eqn:diff}
\end{align}
the last equality holds because $a_1 = b_1+b_2+1$. 
Notice that $ \sum_{2 \leq i_2 < \dots < i_k \leq l}  a_{i_2} \dots a_{i_k} $ is the number of cliques of order $k-1$ in $G[A_2 \cup \dots A_l]$, and $\sum_{2 \leq i_3< \dots < i_{k} \leq l}  a_{i_3} \dots a_{i_k}$ is the number of cliques of order $k-2$ in the same graph. 

Next, we will bound the ratio between these two numbers by a double-counting argument. We will count the number of pairs $(H_1,H_2)$ such that $H_1\subset H_2 \subset G[A_2 \cup \dots A_l]$ and $H_1$ is a $(k-2)$-clique and $H_2$ is a $(k-1)$-clique.
In the graph $G[A_2 \cup \dots A_l]$, each clique of order $k-1$ has $k-1$ cliques of order $k-2$; and each clique of order $k-2$ is in at most $|A_2 \cup \dots \cup A_l| - (k-2) = n-a_1- (k-2) $ cliques of order $k-1$. Thus 
\[
( n- a_1 - (k-2) - a_1) \sum_{2 \leq i_3< \dots < i_{k} \leq l}  a_{i_3} \dots a_{i_k} \geq  (k-1) \sum_{2 \leq i_2 < \dots < i_k \leq l}  a_{i_2} \dots a_{i_k}  . 
\]
Thus inequality in (\ref{eqn:diff}) implies
\[
\sum_{2 \leq i_2 < \dots < i_k \leq l}  a_{i_2} \dots a_{i_k} \geq (b_1 b_2)   \cdot (k-1) \sum_{2 \leq i_2 < \dots < i_k \leq l}  a_{i_2} \dots a_{i_k}  / ( n- (k-2) - a_1). 
\] 
This is equivalent to say
$ n-(k-2)- a_1 > (k-1) (b_1b_2)$,
for any $b_1 + b_2 = a_1 -1$ and $b_1, b_2 >0$. 
By choosing $b_1$ and $b_2$ with difference at most 1, we have $b_1b_2\ge \min \{(\frac{a_1-1}{2})^2, (\frac{a}{2})( \frac{a-2}{2})\}=\frac{a_1^2-2a_1}{4}$.
Thus, we have
$ n-(k-2)- a_1 > (k-1) \frac{(a_1-1)^2-1}{4}.$
Rearranging, we have
$ (a_1-1)^2 < \frac{4n-3k+7-4a_1}{k-1}.$ Proof of this bound for other $a_i$ is same as the proof above for $a_1$.
\end{proof}

 Proofs of the last two parts of Proposition \ref{prop:strucmiddle} are simple computations. We include them in the Appendix A. 

\subsubsection{Proof of Theorem \ref{thm:main}}
We will prove Theorem \ref{thm:main} from Corollary \ref{lem:reduce2}, and we need following lemmas:

\begin{lem}
    Let $h(a,b)=\mathcal{N}_{b}(\mathcal{G}_a \cap \mathcal{D})$ where $a\ge b$. Then, for every non-negative integer $i$, we have $h(a-i,b-i)\le h(a,b)$.
\end{lem}

\begin{proof}
    By definition, $h(a-i, b-i)$ is the maximum number of cliques of order $b-i$ in a dense graph $G$ with $|V(G)| + \omega(G) \leq 2(a-i)-1$. Suppose $G^*$ is the optimizer. Consider $G'$ to be $G^*$ with $i$ extra vertices which form a clique and these $i$ vertices are complete to all the vertices in $G^*$. Then clearly in $G'$,  we have $\omega(G') = \omega(G) +i$, and thus $|V(G')| + \omega(G')  = |V(G)| + i + \omega(G) + i \leq 2(a-i)-1 + 2i = 2a-1$. Moreover, $G'$ is also dense because $|V(G')|-\omega(G')=|V(G)|-\omega(G)$ and $\Delta(G')=\Delta(G)$. Each clique of order $b-i$ in $G^*$ can be extended to a unique clique of order $b$ in $G'$ by extending this clique to the $i$ added vertices. This means the number of cliques of order $b$ in $G'$ is at least the number of cliques of order $b-i$ in $G^*$, which is $h(a-i, b-i)$. On the other hand,  the number of cliques of order $b$ in $G'$ is at most $h(t, k)$. Thus we know $h(t, k) \geq h(t-i, k-i)$. 
\end{proof}

\begin{lem}
    For any $t>k\ge 1$, we have $C^*_t(k-1) \le 4t^2 \cdot C^*_t(k) $. In fact, we will prove $C_t^*(k-1)\le \max\{\lceil \frac{2t-k}{k-1} \rceil^{2}, (2t-2k) \}\cdot C^*_t(k)$.
\end{lem}
\begin{proof}
    By definition, suppose $T^*_t(k-1)=T(2t-\omega^*-1,\omega^*)$ for some $k-1\le \omega^* \le t-1$. 
    
Suppose $\omega^*=k-1$. We may assume that in $T(2t-k,k-1)$, there are $a$ parts of size $\lceil \frac{2t-k}{k-1}\rceil$ and $b$ parts of size $\lfloor \frac{2t-k}{k-1}\rfloor$ with $a+b = k-1$. Thus the number of $(k-1)$-cliques in $T(2t-k,k-1)$ is at most $\lceil \frac{2t-k}{k-1} \rceil^{a} \lfloor \frac{2t-k}{k-1}\rfloor^b$. 
To bound $C^*_t(k)/C^*_t(k-1)$, we construct a $k$-partite graph $H$  with $a-2$ parts of size $\lceil \frac{2t-k}{k-1}\rceil$, $b$ parts of size $\lfloor \frac{2t-k}{k-1}\rfloor$, and three parts of size $1$. 
Then we have $\omega(H)=(a-2)+b+3=k$.
Because $\lceil \frac{2t-k}{k-1} \rceil\ge 2$, we have $|V(H)|\le |V(T(2t-k,k-1))|-4+3=2t-k-1$. 
By Zykov's theorem \cite{zykov}, Tur\'an graph $T(2t-k-1,k)$ maximized the number of the number of $k$-cliques among all the $K_{k+1}$-free graph with at most $2t-k-1$ vertices. Thus, the number of $K_k$ in $T(2t-k-1,k)$ is at least the number of $K_k$ in $H$, which is $\lceil \frac{2t-k}{k-1} \rceil^{a-2} \lfloor \frac{2t-k}{k-1}\rfloor^b$. This means $C^*_t(k-1) \le \lceil \frac{2t-k}{k-1} \rceil^{2} \cdot C^*_t(k)$.

Suppose $k\le \omega^* \le t-1$, because of the structure of Tur\'an graphs, every $(k-1)$-clique $K'$ in $T(2t-\omega^*-1,\omega^*)$ is contained in a $k$-clique $K$; and every $k$-clique in $T(2t-\omega^*-1,\omega^*)$ contains at most $2t-\omega^*-k \leq 2t-2k$ cliques of order $k-1$. This means there are at most $(2t-2k) C^*_t(k)$ cliques of order $k-1$ in $T(2t-\omega^*-1,\omega^*)$ for any $k\le \omega^* \le t-1$.

The two cases above imply  $C^*_t(k-1) \le \max\{\lceil \frac{2t-k}{k-1} \rceil^{2}, (2t-2k) \}\cdot C^*_t(k)\le 4t^2 \cdot C^*_t(k) $.
\end{proof}

\begin{proof}[Proof of Theorem \ref{thm:main}]
Recall $r_0 = 4 t^{1/2}{\log_2}^{1/4} t \ll k<t$. Assume $t$ is sufficiently large.
By Corollary \ref{lem:reduce2} and  inequality (\ref{ineq:crude choose}),
the maximum number of $k$-cliques in a graph on $n$ vertices without a $K_t$-minor is at most
\begin{align*}
  &n \min(r_0, k) \cdot   \left(
\max_{r \leq \min(r_0, k)}   \binom{\beta t \sqrt{\ln t}}{r-1} \cdot \mathcal{N}_{k-r}(\mathcal{G}_{t-r+1}\cap \mathcal{D} )  \right). \\
\le &  n\cdot r_0 \cdot \binom{\beta t \sqrt{\ln t}}{r_0-1} \cdot \max_{1\le r\le r_0} h(t-r+1,k-r)
\le  n\cdot r_0 \cdot \left( \frac{e \beta t \sqrt{\ln t}}{r_0-1}  \right)^{r_0-1} \cdot h(t,k-1) \\
\le & n\cdot r_0 \cdot t ^{r_0-1} \cdot h(t,k-1) 
\le   n \cdot 2^{(r_0\log_2 t)} \cdot  C_t^*(k-1)
\le   n \cdot 2^{(r_0\log_2 t)} \cdot 4t^2\cdot  C_t^*(k).
\end{align*}

The third inequality is true because $r_0> \sqrt{t}> e\beta \sqrt{\ln t}$ when $t$ is large.
Because $|T^*_t(k)|\le 2t$, the bound above is at most $n\cdot 2^{(r_0\log_2 t)} \cdot 8t^3\cdot  \frac{C_t^*(k)}{|T^*_t(k)|} \le n\cdot 2^{2(r_0\log_2 t)} \cdot  \frac{C_t^*(k)}{|T^*_t(k)|}$.
\end{proof}

\begin{rem}
   When $r_0\le k\ll t$, we can improve the bound in Theorem \ref{thm:main} by approximating the maximum point of $\binom{r_0}{r-1} \left( \frac{\beta  t \sqrt{\ln t}}{r_0}  \right)^{r-1}  C_{t-r+1}^*(k-r)$ among $r\le \min (r_0,k)$. More precisely, for any $r_0\leq k <t^{2/3}$, the number of cliques of order $k$ in every graph on $n$ vertices with no $K_t$-minor is at most 
 \[n r_0 \cdot \left( \frac{\beta  t \sqrt{\ln t}}{r_0}  \right)^{r_0} \binom{t - r_0}{k-r_0} 2^{k - r_0}2^{O(\sqrt{\log t})r_0}.
 \]
 It is not hard to show this bound is better than the bounds in Theorem \ref{thm:main}.
\end{rem}

\subsection{Proof of the Main Theorem \ref{thm: summary}}

In this subsection, we will complete the proof of the Main Theorem \ref{thm: summary}. 

\begin{proof}[Proof of Main Theorem \ref{thm: summary}]
We may assume $t$ is sufficiently large. 
When $k=2$, we recall that Thomason \cite{Th1} proved that the number of edges in graphs on $n$ vertices and with no $K_t$-minor is at most $0.32 t\sqrt{\ln t})n$. By Proposition \ref{prop:strucmiddle}, $\frac{C^*_t(k)}{|V(T^*_t(k))|}\ge \frac{(t(t-1)/2)-1}{2t}> t/5$ as $t> 40$. Thus 
$t\sqrt{\ln t}\le (\frac{t}{5})^{1+o_t(1)}\le \left(\frac{C^*_t(k)}{|V(T^*_t(k))|}\right)^{1+o_t(1)}$ which proves the case when $k=2$. 

    When $k\ge 5t/6$, we have $k>2t/3+2\sqrt{t}{\log_2}^{1/4} t$. Then we can apply Corollary \ref{cor: very large k} to prove this theorem for $k$ in this range.
    When $5t/6> k \ge t^{2/3}$, we have $k\gg t^{1/2}{\log_2}^{5/4} t$ and $t-k\ge t/6 \gg t^{1/2}{\log_2}^{5/4} t$. Then we can apply Corollary \ref{cor:main} to prove this theorem for $k$ in this range.
    
    We will finish the proof by showing this theorem is true when $3\le k< t^{2/3}$. 
    
 By Theorem \ref{thm:crudeub}, the number of  $k$-cliques in a $K_t$-minor free graph with $n$ vertices is at most $\binom{\beta  t\sqrt{\ln t}}{k-1} n$,   
 which is at most 
    $n(\frac{t}{k})^{k-1+o_t(1)}$.  
    Because $k< t^{2/3}$, by Proposition \ref{prop:strucmiddle},  $C^*_t(k)\ge \binom{t}{k}2^k (1-4t^{-1/3})^k \geq \binom{t}{k} 1.82^k$ when $t\ge 45^3$. Then  $\frac{C^*_t(k)}{|V(T^*_t(k))|}\ge \frac{(t/k)^k 1.82^k}{2t}\ge (\frac{t}{k})^{k-1}$ because $1.82^k> 2k$ for every $k\ge 3$.
Therefore, when $3\le k< t^{2/3}$, the number of  $k$-cliques in a $K_t$-minor free graph with $n$ vertices is at most $n\binom{\beta  t\sqrt{\ln t}}{k-1}  \le n(\frac{t}{k})^{(k-1)(1+o_t(1))} \le n\cdot \left(\frac{C^*_t(k)}{|V(T^*_t(k))|}\right)^{1+o_t(1)}$. 
      
\end{proof}

\section{Concluding Remarks}
In this paper we studied the problem ${\text{ex}}(n, K_k, K_t\text{-minor})$ and proved an essentially sharp bound, up to $o_t(1)$ in the exponent, for all $k<t$ such that $t-k\gg \log_2 t$. In other words, we showed ${\text{ex}}(n, K_k, K_t\text{-minor}) = C(k,t)^{1+o(1)}n$ where we have a matching lower bound construction which contain $C(k,t)n$ cliques of size $k$ but with no $K_t$-minor. The exact bound in the conjecture of Wood \ref{conj:wood2} remains open.

An analog question is to study the number of $K_k$ in a graph forbidding $K_t$-subdivision instead of $K_t$-minor is also mentioned in this paper. In the case of forbidding $K_t$-subdivision, we even do not know ${\text{ex}}(n, K_2, K_t\text{-subdivision})$.
\begin{ques}
What are the exact values of ${\text{ex}}(n, K_k, K_t\text{-subdivision})$? 
 \end{ques}

\appendix
\section{Completion of Proposition \ref{prop:strucmiddle} (Proposition \ref{prop:Tstarstructure})} \label{section:stru}

Now we prove Proposition \ref{prop:Tstarstructure}. 
We first give an upper bound on $C_t^*(k)$ which is the number of cliques of order $k$ in $T^*_t(k)$, and also the optimal objective function value for (\ref{eqn:ckG''2}). 
\begin{claim}\label{claim:simpleckbound}
\[ C_t^*(k) \leq \binom{t-1}{k} 2^k. \]
\end{claim}
\begin{proof}
Suppose the optimal graph, which is a balanced complete multipartite graph by Lemma \ref{lem:bcm}, has $x$ parts of order $a \geq 1$ and $y$ parts of order $a+1$. Thus Claim \ref{claim:full} implies
$
 |V(G)| + \omega(G) = (a+1)x + (a+2)y = 2t-1. $%\label{eqn:axy}
%\end{equation}
This implies 
\begin{equation}
x+ y \leq  \lfloor(2t-1)/(a+1) \rfloor. \label{eqn:xaddy}
\end{equation} 
The number of cliques of order $k$ in this graph is 
\begin{equation}
\sum_{i=0}^k \binom{x}{i} \binom{y}{k-i} a^i (a+1)^{k-i}. \label{eqn:ckoriginal}
\end{equation}
We can upper bound the above quantity by
\begin{equation}
\sum_i \binom{x}{i} \binom{y}{k-i} (a+1)^i (a+1)^{k-i} = \binom{x+y}{k} (a+1)^k.  \label{eqn:ckcountbound}
\end{equation}

By (\ref{eqn:xaddy}),  the number of cliques of order $k$ is at most 
\[
\binom{x+y}{k} (a+1)^k \leq \binom{\lfloor \frac{2t-1}{a+1} \rfloor}{k} (a+1)^k 
:= f(a+1).\]
It can be checked that the function $f(a+1)$ is monotone decreasing in $a$. Thus the largest value is chosen when $a = 1$, which is $\binom{t-1}{k} 2^k$. 
\end{proof}

In fact, when $k \ll t$, the above upper bound is essentially correct. We construct a lower bound for $c_k(T_t^*(k))$ which almost matches the upper bound in Claim \ref{claim:simpleckbound}. 
\begin{lem}\label{lem:ckapprox}

 \[ C_t^*(k)\geq \binom{t-1}{k} \max\left(1, \left( 2-4\sqrt{ {2k}/{t}}  \right)\right)^k.\]
This bound can be achieved by considering $T(n, w)$ where $w = \sqrt{kt/2}$ and $n = 2t-1 - w$.

\end{lem}
\begin{proof}
Notice by considering a clique on $t-1$ vertices, we have $C_t^*(k)\ge \binom{t-1}{k}$. Given $n, w$, each part of $T(n, w)$ has size between $n/w-1$ and $n/w+1$.  Thus the number of cliques of size $k$ in $T(n,w)$ is at least
\begin{equation}
 \binom{w}{k}(n/w -1)^k \geq \binom{w}{k} (n/w)^k \left(   \frac{n/w-1}{n/w} 
\right)^k =  \frac{\prod_{i=0}^{k-1} (w-i)}{k!} (n/w)^k \left(   1- w/n
\right)^k. \label{eqn:coptimalapprox}
\end{equation}
Plugging in  $n = 2t-w-1$, the right hand side of (\ref{eqn:coptimalapprox}) is 
\[ \frac{\prod_{i=0}^{k-1} (w-i)}{k!} ((2t-w-1)/w)^k \left(   1- w/n
\right)^k = (2t)^k  \frac{\prod_{i=0}^{k-1} (w-i)}{k!} \left(\frac{1}{w} - \frac{1}{2t} - \frac{1}{2tw}\right)^k \left(   1- w/n
\right)^k
\]

We know $i \leq k-1$ and $k \leq w < t \leq n$, thus
\begin{align*}
 & (w-i) \left(\frac{1}{w} - \frac{1}{2t} - \frac{1}{2tw}\right) \left(   1- w/n
\right)  
 >   (w-k) \left(\frac{1}{w} - \frac{1}{2t} - \frac{1}{2tw}\right) \left(   1- w/n
\right) \\ 
= & 1 - \frac{w}{2t} - \frac{1}{2t} - \frac{k}{w} + \frac{k}{2t} + \frac{k}{2tw} -  \frac{w}{n} + \frac{w^2}{2nt} + \frac{w}{2tn} + \frac{k}{n} - \frac{kw}{2tn} - \frac{k}{2tn} 
\geq  1- \frac{2w}{t}  - \frac{k}{w} .
\end{align*}
To maximize this lower bound, we choose $w = \sqrt{tk/2}$. Then we have $1- \frac{2w}{t}  - \frac{k}{w} = 1 - 2\sqrt{\frac{2k}{t}}$.
Therefore the right hand side of (\ref{eqn:coptimalapprox}) is at least
%\begin{equation} 
$\frac{(2t)^k}{k!} \left(    1- 2\sqrt{2k/t}   \right)^k \geq 2^k \binom{t}{k} \left(    1- 2\sqrt{2k/t}   \right)^k. %\label{eqn:bound1}
$ %\end{equation}
\end{proof}

We can now prove asymptotically the number of parts in the optimal graph  $T_t^*(k)$ by comparing the number of $k$-cliques in $T(n,w)$ to the upper and lower bounds above. It turns out that the construction in Lemma \ref{lem:ckapprox} is of the correct order. To be more specific, the Tur\'an graph $T(n,w)$ which is $T_t^*(k)$  is such that $w = \Theta(\sqrt{tk})$. 
\begin{lem}[Restatement of the first part in Proposition \ref{prop:Tstarstructure}] \label{lem:tstar}
For any $k \leq t$, the optimal graph $T_t^*(k)$ has $\omega$ parts where $\sqrt{tk}/4\le \omega \le 10 \sqrt{tk}$.
\end{lem}
\begin{proof}
Again assume the optimal graph has $x$ parts of order $a\geq 1$ and $y$ parts of order $a+1$.

Given $n, w$, each part of $T(n, w)$ has size between $n/w-1$ and $n/w+1$.  Thus by the AM-GM inequality, the number of cliques of size $k$ in $T(n,w)$, which is $\binom{n}{k}_w$, satisfies
\begin{equation}
 \binom{n}{k}_w\leq \binom{w}{k} (n/w)^k =  \frac{\prod_{i=0}^{k-1} (w-i)}{k!} (n/w)^k. \label{eqn:coptimalapprox2}
\end{equation}
Plugging in  $n = 2t-w-1$, the right hand side of equation (\ref{eqn:coptimalapprox2}) is
\begin{align}
& \frac{\prod_{i=0}^{k-1} (w-i)}{k!} ((2t-w-1)/w)^k  = (2t)^k  \frac{\prod_{i=0}^{k-1} (w-i)}{k!}  \left(\frac{1}{w} - \frac{1}{2t} - \frac{1}{2tw}\right)^k \\
 \leq & (2t)^k  \frac{(w - (k-1)/2)^k}{k!} \left(\frac{1}{w} - \frac{1}{2t} - \frac{1}{2tw}\right)^k.
\end{align}
The last inequality is by the fact that $(w -i) (w- (k-1-i)) \leq  (w - (k-1)/2)^2$ for all $0 \leq i \leq k-1$. 
We know $i \leq k-1$ and $k \leq w < t \leq n$, thus
\begin{align*}
 & (w-(k-1)/2) \left(\frac{1}{w} - \frac{1}{2t} - \frac{1}{2tw}\right) \\
 %>  & (w-k) \left(\frac{1}{w} - \frac{1}{2t} -\frac{1}{2tw}\right) \left(   1- w/n\right) \\ 
= & 1 - \frac{w}{2t} - \frac{1}{2t} - \frac{(k-1)/2}{w} + \frac{(k-1)/2}{2t} + \frac{(k-1)/2}{2tw} \\
< & 1 - \frac{w}{2t} - \left( \frac{1}{2t} - \frac{(k-1)/2}{2tw}\right) +  \left(-\frac{(k-1)/2}{w}  + \frac{(k-1)/2}{2w} \right) \  \  \ \ \ \ \  \left( \text{  since } \frac{(k-1)/2}{2t} < \frac{(k-1)/2}{2w}\right) \\
\leq & 1 - \frac{w}{2t} - \frac{(k-1)/2}{2w}  \ \ \  \ \ \ \ \  \left( \text{  since } \frac{(k-1)/2}{w} \leq 1 \right).
\end{align*}
It can be seen that the maximum of the right-hand side is achieved when $w = \sqrt{(k-1)t/2}$. On the other hand, if $w \geq 10\sqrt{kt} > 10 \sqrt{(k-1)t}$ and $k \geq 2$, 
the right-hand side is at most 
\[
1 - \frac{ 10 \sqrt{(k-1)t}}{2t} - \frac{(k-1)/2}{20 \sqrt{(k-1)t}} < 1- 5\sqrt{(k-1)/t} < 1- 2\sqrt{2k/t}.
\]
This means that if $w \geq 10\sqrt{kt} > 10 \sqrt{(k-1)t}$, then the objective function (\ref{eqn:coptimalapprox2}) is at most 
\[\frac{ (2t)^k}{k!} \left( 1- 5\sqrt{(k-1)/t}  \right)^k < \frac{ (2t)^k}{k!} \left( 1- 2\sqrt{2k/t}  \right)^k\]
where the right-hand side is a lower bound for the optimal objective function as has been proved in the Lemma \ref{lem:ckapprox}.
This means in the optimal graph, 
\begin{equation}
 w \leq 10\sqrt{kt}. \label{eqn:wsmall}
 \end{equation}

 On the other hand, by Lemma \ref{lem:smalla}, we know the size of each part $a_i$ in $T_t^*(k)$ satisfies
$(a_i-1)^2 < \frac{4n-3k+7-4a_i}{k-1}
$ or $a_i < 3$. Let $a_1$ be the size of largest part of $T_t^*(k)$, then we have $a_1=\lceil \frac{n}{w} \rceil \ge \frac{n}{w}$.
This means the number of parts $w$ in $T_t^*(k)$ satisfies $n/w \le a_1 <3$ or 
\[ (n/w-1)^2 \leq (a_1-1)^2 \le  \frac{4n-3k+7-4a_1}{k-1} < \frac{4n-3k}{k}.\]
Thus $w> n/3$ or $\frac{n}{w} \le  \sqrt{\frac{4n-3k}{k}}+1$.
Since $t \leq n \leq 2t$ and $k \leq n$, we have $w> t/3$ or $\frac{n}{w} \le  2\sqrt{\frac{4n-3k}{k}}$. When the second case happens, we have 
\[
w\ge \frac{1}{2}\sqrt{\frac{n^2k}{4n-3k}} \ge \frac{1}{2}\sqrt{\frac{n^2k}{4n}} \ge \frac{1}{4}\sqrt{nk} \ge \frac{1}{4}\sqrt{tk}
\]
Thus, we have $w\ge \min \{t/3, \sqrt{tk}/4\}=\sqrt{tk}/4$. Therefore combining with (\ref{eqn:wsmall}), we proved that in the optimal graph $T_t^*(k)$, the number of parts is of the order $\Theta(\sqrt{tk})$. 
\end{proof}

\section{Disproof of Conjecture \ref{conj:wood2}} \label{section: disproof} 
We now give a construction and show Wood's Conjecture \ref{conj:wood2} does not hold for $\lambda \leq 0.553$. 
\begin{thm}\label{thm:woodconj2false}
Let $k = \lambda t$ where $\lambda \leq 0.553$. Then when $t$ is large, there exists a graph on $n$ vertices without $K_t$ as a minor, and the  number of cliques of order $k$ in this graph is strictly larger than
$\binom{t-2}{k-1}n$.
\end{thm}
\begin{proof}
Consider a graph $G$ on $n$ vertices which is a union of the complement of a perfect matching on $2(t-1)/3$ edges. We can assume $t \equiv 1 \mod 3$ and $n$ is divisible by $4(t-1)/3$. Thus by Lemma \ref{lem:densehnumber}, the largest clique minor order in $G$ is $t-1$. On the other hand, the number of cliques of order $k$ in $G$ is 
\[
\binom{2(t-1)/3}{k}2^k \cdot \frac{n}{4(t-1)/3}.\]
The last term $\frac{n}{4(t-1)/3}$ is the number of copies of the graph which is the complement of a perfect matching. Each copy has exactly $\binom{2(t-1)/3}{k}2^k$ cliques of order $k$; this is because each edge in the matching can contribute to at most one vertex in the clique.

We want to show that when $t$ is large, 
\begin{equation} \binom{2(t-1)/3}{k}2^k \cdot \frac{n}{4(t-1)/3} > \binom{t-2}{k-1}n.  \label{eqn:goal}
\end{equation}
Assume $k = \lambda t$ where $1/3 < \lambda < 0.553$. Then by Stirling's formula applied to the binomial coefficient, letting $h(x) = x \log_2 x$, the left-hand side of (\ref{eqn:goal}) is at least  
\[ n c_1 \sqrt{  \frac{2t/3}{k(2t/3-k)}    }2^{h(2(t-1)/3) - h(k) - h(2(t-1)/3 -k) + k},\]
where $c_1$ is some absolute constant. 
Similarly, the right hand side of (\ref{eqn:goal})  is at most 
\[ n c_2 \sqrt{  \frac{t}{k(t-k)}    }2^{h(t-2) - h(k-1) - h(t-k-1)},\] where again $c_2$ is some absolute constant.
It suffices to show that for each $\lambda$, there is some constant $\epsilon$ such that 
\begin{equation}
 h(2(t-1)/3) - h(k) - h(2(t-1)/3 -k) + k > \epsilon t + h(t-2) - h(k-1) - h(t-k-1). \label{eqn:goal2}
 \end{equation}
If this is the case, then to prove (\ref{eqn:goal}), it suffices to show
\[ n c_1 \sqrt{  \frac{2t/3}{k(2t/3-k)}    } 2^{\epsilon t} > n c_2 \sqrt{  \frac{t}{k(t-k)}    }.\] 
This clearly holds when $k = \lambda t$ where $\lambda$ is fixed and $t$ sufficiently large. Thus it suffices to prove (\ref{eqn:goal2}) for some $\epsilon >0$. 

As $h'(x) = \log x + 1/\ln(2)$,  for $a > b>1$, $0 \leq h(a) - h(b) \leq (b-a ) (\log a +1/\ln(2))$. If $b-a \ll a$, we will have  when $a$ sufficiently large, 
$ h(b) = h(a) + O(\log a)$.
Thus to prove (\ref{eqn:goal2}) for some $\epsilon>0$, it suffices to prove there exists a constant $\epsilon' >0$ such that when $t$ is sufficiently large, 
\[ h(2t/3) - h(k) - h(2t/3 -k) + k > \epsilon' t + h(t) - h(k) - h(t-k).\]
Removing $h(k)$ from both sides, 
it suffices to prove $ h(2t/3) - h(2t/3 -k) + k > \epsilon' t + h(t) - h(t-k)$.
Using $k = \lambda t$, notice 
\begin{align*} & h(t) - h(t-k) = t \log t - (t-k) \log(t-k) 
=  t \log t - (t-k) \log(t (1-k/t)) \\
 = & t\log t - (t-k) \log t - (t-k) \log(1-\lambda) = \lambda t \log t - (1-\lambda) t \log(1-\lambda).
\end{align*}

Similarly, for the left hand side,
\begin{align*} & h(2t/3) - h(2t/3-k)=
 \lambda t \log (2 t/3) - (2/3-\lambda) t \log(1-3\lambda/2) \\
  =& \lambda t \log (t) + \lambda t \log(2/3)  - (2/3-\lambda)t \log(3/2) - (2/3 - \lambda)t \log(2/3 - \lambda) \\
  = &  \lambda t \log (t) + 2t/3\log(2/3) - (2/3 - \lambda)t \log(2/3 - \lambda)
\end{align*}
Therefore we want to prove 
\[ 
- (1-\lambda) t \log(1-\lambda) + \lambda t \log t+  \epsilon' t < k+ \lambda t \log t +2t/3\log(2/3) - (2/3 - \lambda)t \log(2/3 - \lambda).\]
Removing $ \lambda t \log t$ from both ends, and dividing both sides by $t$, it is equivalent to show
$  \epsilon'- h(1-\lambda)   < \lambda + h(2/3) - h(2/3 - \lambda)$.
The function $f(\lambda) = \lambda + h(2/3) - h(2/3 - \lambda) + h(1-\lambda)$
is strictly positive for $\lambda \leq 0.553$, which means the existence of positive $\epsilon'$. 
\end{proof}

\end{document}